\documentclass{article}

\usepackage{arxiv}

\usepackage[utf8]{inputenc} 
\usepackage[T1]{fontenc}    

\usepackage[numbers]{natbib}
\usepackage{doi}
\RequirePackage{amsthm,amsmath,amsfonts,amssymb,bm}
\usepackage[dvipsnames]{xcolor}
\definecolor{lightblue}{HTML}{00F9DE}

\RequirePackage{graphicx,url}
\usepackage{appendix}
\usepackage{fixltx2e}
\usepackage{float}
\usepackage{tikz}
\usetikzlibrary{fit,positioning}
\usetikzlibrary{bayesnet}
\usetikzlibrary{shapes,arrows}
\tikzstyle{block} = [draw, fill=white, rectangle, 
    minimum height=2cm, minimum width=0.5cm]
\numberwithin{equation}{section}
\newtheorem{theorem}{Theorem}[section]

\newtheorem{proposition}[theorem]{PROPOSITION}
\newtheorem{corollary}[theorem]{Corollary}
\theoremstyle{remark}
\newtheorem{definition}[theorem]{Definition}
\newtheorem*{example}{Example}

\newtheorem{assumption}[theorem]{ASSUMPTION}
\newcommand{\indep}{\perp \!\!\! \perp}

\DeclareFontFamily{OT1}{pzc}{}
\DeclareFontShape{OT1}{pzc}{m}{it}{<-> s * [1.10] pzcmi7t}{}
\DeclareMathAlphabet{\mathpzc}{OT1}{pzc}{m}{it}

\title{Hierarchical Causal Analysis of Natural Languages on a Chain Event Graph}

\date{} 					

\author{{\hspace{1mm}Xuewen Yu}\thanks{Xuewen Yu is funded by
	\emph{EPSRC} and the \emph{Statistics Department of the University of Warwick}.} \\
	Department of Statistics\\
	 University of Warwick\\
	Coventry, CV4 7AL, UK \\
	\texttt{xuewen.yu@warwick.ac.uk} \\
	\And
	{\hspace{1mm}Jim Q. Smith} \thanks{Professor Jim Q. Smith is supported by the
	\emph{Alan Turing Institute} and \emph{EPSRC} with grant number \emph{EP/K039628/1}.}\\
	Department of Statistics\\
	University of Warwick\\
	Coventry, CV4 7AL, UK \\
	The Alan Turing Institute\\
	London, NW1 2DB, UK\\
	\texttt{j.q.smith@warwick.ac.uk} \\
}

\renewcommand{\shorttitle}{Hierarchical Causal Analysis}

\begin{document}
\maketitle

\begin{abstract}
	Various graphical models are widely used in reliability to provide a qualitative description of domain experts' hypotheses about how a system might fail. Here we argue that the semantics developed within standard causal Bayesian networks are not rich enough to fully capture the intervention calculus needed for this domain and a more tree-based approach is necessary. We instead construct a Bayesian hierarchical model with a chain event graph at its lowest level so that typical interventions made in reliability models are supported by a bespoke causal calculus. We then demonstrate how we can use this framework to automate the process of causal discovery from maintenance logs, extracting natural language information describing hypothesised causes of failures. Through our customised causal algebra we are then able to make predictive inferences about the effects of a variety of types of remedial interventions. The proposed methodology is illustrated throughout with examples drawn from real maintenance logs. 
\end{abstract}



\section{Introduction}
In this paper we develop a new class of Bayesian graphical model whose semantics are customised to supporting a causal analysis for reliability. 
More specifically we build a theoretical framework that transforms standard methods of causal analyses so these can be applied to the systematic extraction of causal relationships of particular interest from natural language descriptions embedded in engineering reports about what they believe might explain malfunction or potential malfunction they observe. Within this domain such reports provide critical information about causal explanations of different failure events.

Despite the popularity of the Bayesian network (BN) framework for exploring causal relationships like this, Shafer\citep{shafer1996} and others have argued that event tree based inference provides an even more flexible and expressive graph from which to explore causal relationships. Within the domain of reliability where the focus of inference is on explaining and repairing failure incidents in a system, the use of trees and their derivative depictions such as chain event graphs (CEG) \citep{Collazo17,Collazo2018, Collazo18} are especially efficacious. Here a collection of paths on these graphs are used to explain the unfolding of events that might have led to the fault. One key advantage of the CEG is that unlike the BN the asymmetric unfoldings of the process can be directly represented by its topology \citep{Barclay2013}. This provides a vehicle through which to define a direct mapping from the collections of features found by natural language processors on to a probability model that faithfully represents their explanatory statements. In Section \ref{sec:ceg} we illustrate how this map can be constructed for a given system.

Recently natural language processing has been developed so that it is able to extract from natural languages texts various primitive causal asserts in terms of a symbolic logical dictionary \citep{Chambers2014,Mirza2016}. However to our knowledge no methodology has yet been devised that extracts the features encoded within the text that directly map on to a causal algebra, which like the BN can then be associated with families of probability models of both an idle system and the system under intervention. For example, given a document recording that ``\textit{the seal deterioration caused oil leak in the conservator - topping up oil}'', the proposed framework enables us to analyse the root cause of this fault and predict the effect of the remedial act of adding oil on the equipment. Figure \ref{fig:framework} sketches the architecture behind the novel hierarchical causal model. We call this a \textit{GN-CEG} causal model. It has a BN at its surface level and a CEG at its deeper level. The CEG depicting the causality embedding provides the framework for a probability model faithful to an engineer's hypothesises about what might have happened. Therefore it can be used to answer and evaluate quantitatively dependency queries about any malfunction. In Section \ref{sec:nlp}, we explain in detail how to preprocess the texts data and how to construct the novel hierarchical causal framework while Section \ref{sec:rc} formalises the nested causality.  

\begin{figure}
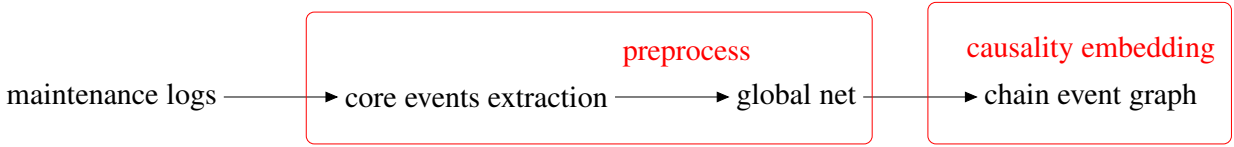

    \centering
    \resizebox{\linewidth}{!}{
    \tikz{
    \node[font=\fontsize{15}{0}\selectfont](n1){\text{maintenance logs}};
    \node[right = 2cm of n1,font=\fontsize{15}{0}\selectfont](n2){\text{core events extraction}};
    \node[right = 2cm of n2,font=\fontsize{15}{0}\selectfont](n3){\text{global net}};
    \node[right = 2cm of n3,font=\fontsize{15}{0}\selectfont](n4){\text{chain event graph}};
    \node[above right= 0.1cm and 0.01cm of n2,font=\fontsize{15}{0}\selectfont,color=red](n5){\text{preprocess}};
    \node[above  = 0.1cm of n4,font=\fontsize{15}{0}\selectfont,color=red](n6){\text{causality embedding}};
    \edge{n1}{n2};
    \edge{n2}{n3};
    \edge{n3}{n4};
    \plate[inner sep=0.35cm, xshift=-0.2cm, yshift=0.12cm,color=red]{plate}{(n2)(n3)(n5)}{};
    \plate[inner sep=0.35cm, xshift=-0.2cm, yshift=0.12cm,color=red]{plate2}{(n4)(n6)}{};
    }
    
    }
    \caption{The proposed hierarchical causal framework.}
    \label{fig:framework}
\end{figure}

Thus we show that it is possible to build an event tree based probabilistic intervention calculus - parallel to those devised by Spirtes \textit{et.al} \citep{Spirtes2000} and Pearl \citep{Pearl2000} - to translate the natural language explanations of engineers about the likely causes of a failure - into families of probability models.
The causal algebra needed for this domain is rather different from more familiar ones. For example when components of a system are replaced, they will behave as if starting from new rather than starting from a point of embedded usage. The latter would be what we would need to assume were we to use conventional algebras. So we need to develop rather different formulae for causal predictions from those devised by Pearl for BNs \citep{Pearl2000} and their translation into tree based causal inference \citep{Thwaites10,Thwaites13}. 

Furthermore, because of the routine use of the terms within reliability analyses - we have discovered the usefulness of an object called a \textit{remedy}: a collection of acts designed to attempt to rectify the root cause of a failure incident. As illustrated above, potential remedies are especially useful objects in this context because we find that they can often be extracted directly from the natural language texts engineers routinely provide. Interventions associated with remedial acts are in practice nearly always very specific types of non-atomic (not singular)\citep{Pearl2000,Thwaites13} interventions. So a semantic needs to be developed to properly encode these. In Section \ref{sec:intervention} we provide these semantics to elegantly express these naturally arising composites.

One critical output of such a causal analysis is to determine when, under the hypotheses we use, we are able to make complete probabilistic predictions of the effects of an intervention when the system can only be partially observed. In Section \ref{sec:identifiability} we prove a number of analogues of so called back-door theorems \citep{Pearl2000,Thwaites10,Thwaites13} developed in other settings customised to the particular needs of this domain.

The maintenance logs may not provide complete information about a failure or deteriorating process and its maintenance procedures. Therefore when translating causal events expressed by natural languages to a causal CEG, the missingness must be captured. In Section \ref{sec:missing}, we explain graphically and mathematically how to formally manage different types of missing information within the new causal analyses and develop an extension of the back-door criterion\citep{Saadati2019} for addressing the missingness mechanism.

We begin the paper by recalling some of the definitions of the CEG we use at the deepest level of our hierarchical model.

\section{The causal CEG}\label{sec:ceg}
Let an event tree $T=(E_T,V_T)$ have an associated probability vector $\bm{\theta}_T=(\bm{\theta}_v)_{v\in V_T}$, called the \textit{primitive probabilities} of the tree \citep{Thwaites10}. Any edge emanating from $v\in V_T$, denoted by  $e_{v,v'}\in E_T$, is annotated by a probability $\theta_{v,v'}\in \bm{\theta}_v$. Let $ch(v)=\{v'\in V_T|e_{v,v'}\in E_T\}$ denote the set of children of vertex $v$. The pair $(T,\bm{\theta}_T)$ is called a \textit{probability tree} \citep{Gorgen2018} when $\sum_{v'\in ch(v)}\theta_{v,v'}=1$ and $\theta_{v,v'}\in (0,1)$ for all $v\in V_T$.

Two vertices $v$ and $w$ are said to be at the same \textit{stage} if $\bm{\theta}_v=\bm{\theta}_w$ up to a permutation of their components \citep{Collazo2018}. Each stage and vertices in the same stage are assigned a unique colour. 
The resulting coloured tree $(T,\bm{\theta}_T)$ is called a \textit{staged tree}.  

Let $T(v)\subseteq T$ represent the full subtree which is rooted at $v\in V_T$ and comprises all the $v-$to-leaf paths in $T$. Then two vertices $v,w\in V_T$ are in the same \textit{position} if $(T(v),\bm{\theta}_{T(v)})$ and $(T(w),\bm{\theta}_{T(w)})$ have isomorphic graphs with all their associated edge probabilities also equated \citep{Barclay2015, Collazo2018}.

\begin{definition} A \textit{Chain Event Graph} (CEG) $(C,\bm{\theta}_T)$ is constructed from a staged tree $(T,\bm{\theta}_T)$. The graph of a CEG is represented by $C=(V_C,E_C)$. The vertex set $V_C$ is given by the set of positions in the underlying staged tree. Every position inherits its colour from the staged tree. If there exist edges $e_{v,v'}$,$e'_{w,w'}\in E_T$ and $v,w$ are in the same position, then there exist corresponding edges $f, f'\in E_C$. If also $v',w'$ are in the same position, then $f=f'$. The probabilities $\theta_f$ of edge $f\in E_C$ in the new graph are inherited from the corresponding edges $e\in E_T$ in the staged tree \citep{Barclay2015, Collazo2018}.
\end{definition}

A \textit{floret} on a CEG $C$ is a pair $F(v)=(v,E(v))$, where $v\in V_C$ is a vertex of $C$ and $E(v)=\{(v,v')\in E_C|v'\in ch(v)\}$ is the set of emanating edges from $v$ \citep{ Collazo2018}. Let $\Lambda(C)$ denote the set of  root to sink paths on $C$ and $\lambda\in \Lambda(C)$ denote a single path.

\begin{figure}
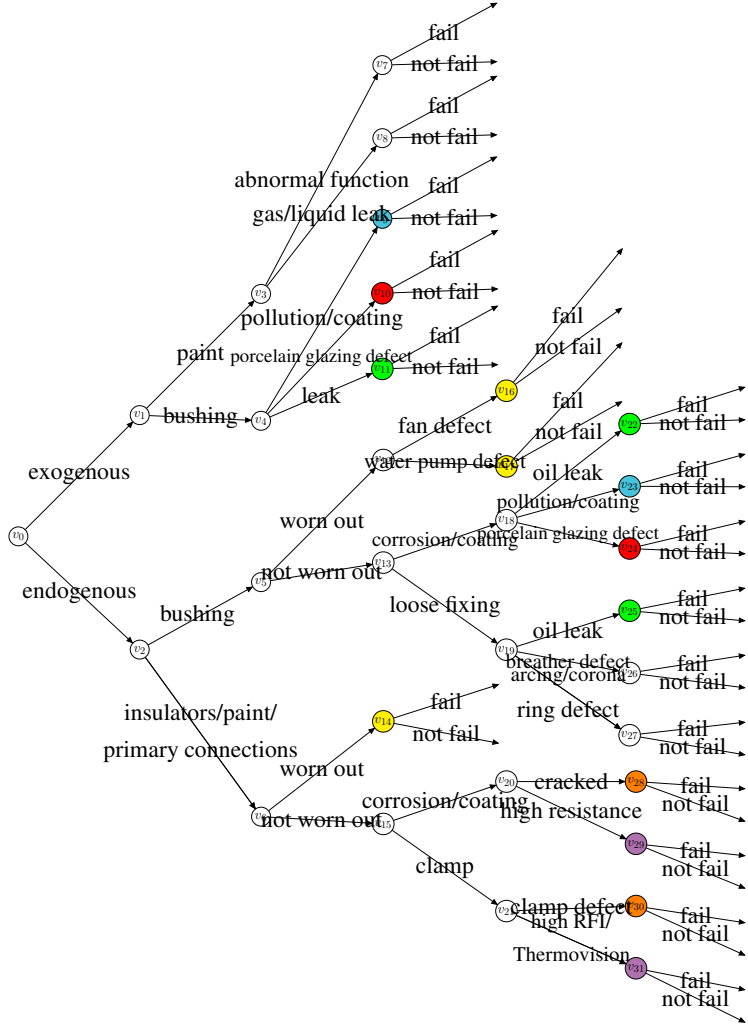

    \centering
    \resizebox{0.6\linewidth}{!}{
    \tikz{
    \node[latent,font=\fontsize{15}{0}\selectfont](v0){$v_0$};
    \node[latent, above right = 4cm and 4cm of v0,font=\fontsize{15}{0}\selectfont](v1){$v_1$};
    \node[latent,below = 8cm of v1,font=\fontsize{15}{0}\selectfont](v2){$v_2$};
    \node[latent,above right = 4cm and 4cm of v1,font=\fontsize{15}{0}\selectfont](v3){$v_3$};
    \node[latent,below = 4cm of v3,font=\fontsize{15}{0}\selectfont](v4){$v_4$};
    \node[latent,above right = 2cm and 4cm of v2,font=\fontsize{15}{0}\selectfont](v5){$v_5$};
    \node[latent,below = 8cm of v5,font=\fontsize{15}{0}\selectfont](v6){$v_6$};
    \node[latent,above right = 8cm and 4cm of v3,font=\fontsize{15}{0}\selectfont](v7){$v_7$};
    \node[latent,below = 2cm of v7,font=\fontsize{15}{0}\selectfont](v8){$v_8$};
    
     \node[latent,above right = 7cm and 4cm of v4,font=\fontsize{15}{0}\selectfont,fill=SkyBlue](v9){$v_9$};
    \node[latent,below = 2cm of v9,font=\fontsize{15}{0}\selectfont,fill=red](v10){$v_{10}$};
    \node[above right = 2cm and 4cm of v7](s1){};
    \node[below = 2cm of s1](s2){};
    \node[above right = 2cm and 4cm of v8](s3){};
    \node[below = 2cm of s3](s4){};
    
    \node[above right = 2cm and 4cm of v9](s5){};
    \node[below = 2cm of s5](s6){};
    \node[above right = 2cm and 4cm of v10](s7){};
    \node[below = 2cm of s7](s8){};
    
    \node[latent,below = 2cm of v10,font=\fontsize{15}{0}\selectfont,fill=green](v66){$v_{11}$};
    \node[above right = 2cm and 4cm of v66](s9){};
    \node[below = 2cm of s9](s10){};
    
    \node[latent, above right = 4cm and 4cm of v5,font=\fontsize{15}{0}\selectfont](v11){$v_{12}$};
    \node[latent,below = 3cm of v11,font=\fontsize{15}{0}\selectfont](v12){$v_{13}$};
    
    \node[latent, above right = 3cm and 4cm of v6,font=\fontsize{15}{0}\selectfont,fill=yellow](v13){$v_{14}$};
    \node[latent,below = 3cm of v13,font=\fontsize{15}{0}\selectfont](v14){$v_{15}$};

   \node[latent,above right = 2cm and 4cm of v11,font=\fontsize{15}{0}\selectfont,fill=yellow](v15){$v_{16}$};;
   \node[latent,below = 2cm of v15,font=\fontsize{15}{0}\selectfont,fill=yellow](v16){$v_{17}$};
    
    \node[latent,above right =1cm and 4cm of v12,font=\fontsize{15}{0}\selectfont](v17){$v_{18}$};
    \node[latent,below = 4cm of v17,font=\fontsize{15}{0}\selectfont](v18){$v_{19}$};
    
    \node[above right = 1cm and 4cm of v13](v19){};
    \node[below = 2cm of v19](v20){};
    
     \node[latent,above right =  1cm and 4cm of v14,font=\fontsize{15}{0}\selectfont](v21){$v_{20}$};
    \node[latent,below = 4cm of v21,font=\fontsize{15}{0}\selectfont](v22){$v_{21}$};
    
    \node[above right = 5cm and 4cm of v15](r1){};
    \node[below = 2cm of r1](r2){};
    \node[below = 1cm of r2](r3){};
    \node[below = 2cm of r3](r4){};
    
    \node[latent,above right = 3cm and 4cm of v17,font=\fontsize{15}{0}\selectfont,fill=green](v23){$v_{22}$};
    \node[latent,below = 1.5cm of v23,font=\fontsize{15}{0}\selectfont,fill=SkyBlue](v24){$v_{23}$};
    \node[latent,below = 1.5cm of v24,font=\fontsize{15}{0}\selectfont,fill=red](v25){$v_{24}$};
    \node[latent,below = 1.5cm of v25,font=\fontsize{15}{0}\selectfont,fill=green](v26){$v_{25}$};
    \node[latent,below = 1.5cm of v26,font=\fontsize{15}{0}\selectfont](v27){$v_{26}$};
    \node[latent,below = 1.5cm of v27,font=\fontsize{15}{0}\selectfont](v28){$v_{27}$};
    
    \node[latent, right = 4cm of v21,font=\fontsize{15}{0}\selectfont,fill=orange](v29){$v_{28}$};
    \node[latent,below = 1.5cm of v29,font=\fontsize{15}{0}\selectfont,fill=Orchid](v30){$v_{29}$};
    \node[latent,below = 1.5cm of v30,font=\fontsize{15}{0}\selectfont,fill=orange](v31){$v_{30}$};
    \node[latent,below = 1.5cm of v31,font=\fontsize{15}{0}\selectfont,fill=Orchid](v32){$v_{31}$};
    
    \node[above right = 1cm and 4cm of v23](h1){};
    \node[below = 1cm of h1](h2){};
    
    \node[below = 1cm of h2](h3){};
    \node[below = 1cm of h3](h4){};
    
    \node[below = 1cm of h4](h5){};
    \node[below = 1cm of h5](h6){};
    
    \node[below = 1cm of h6](h7){};
    \node[below = 1cm of h7](h8){};
    
    \node[below = 1cm of h8](h9){};
    \node[below = 1cm of h9](h10){};
    
    \node[below = 1cm of h10](h11){};
    \node[below = 1cm of h11](h12){};
    
    \node[below = 1cm of h12](h13){};
    \node[below = 1cm of h13](h14){};
    \node[below = 1cm of h14](h15){};
    \node[below = 1cm of h15](h16){};
    
    \node[below = 1cm of h16](h17){};
    \node[below = 1cm of h17](h18){};
    \node[below = 1cm of h18](h19){};
    \node[below = 1cm of h19](h20){};

    \path[->,draw]
   (v0)edge node[scale=2.5]{\text{exogenous}}(v1)
   (v0)edge node[scale=2.5]{\text{endogenous}}(v2) 
   (v1)edge node[scale=2.5]{\text{paint}}(v3)
   (v1)edge node[scale=2.5]{\text{bushing}}(v4)
   (v2)edge node[scale=2.5]{\text{bushing }}(v5)
    (v2)edge[above] node[scale=2.5]{\text{insulators/paint/}}(v6)
    (v2)edge[below] node[scale=2.5]{\text{primary connections}}(v6)
    (v3)edge node[scale=2.5]{\text{abnormal function}}(v7)
    (v3)edge node[scale=2.5]{\text{gas/liquid leak}}(v8)
    (v4)edge node[scale=2.5]{\text{pollution/coating}}(v9)
    (v4)edge node[scale=2]{\text{porcelain glazing defect}}(v10)
    (v7)edge node[scale=2.5]{\text{fail}}(s1)
    (v7)edge node[scale=2.5]{\text{not fail}}(s2)
    (v8)edge node[scale=2.5]{\text{fail}}(s3)
    (v8)edge node[scale=2.5]{\text{not fail}}(s4)
    (v9)edge node[scale=2.5]{\text{fail}}(s5)
    (v9)edge node[scale=2.5]{\text{not fail}}(s6)
    (v10)edge node[scale=2.5]{\text{fail}}(s7)
    (v10)edge node[scale=2.5]{\text{not fail}}(s8)
    (v5)edge node[scale=2.5]{\text{worn out}}(v11)
    (v5)edge node[scale=2.5]{\text{not worn out}}(v12)
    (v6)edge node[scale=2.5]{\text{worn out}}(v13)
    (v6)edge node[scale=2.5]{\text{not worn out}}(v14)
    (v11)edge node[scale=2.5]{\text{fan defect}}(v15)
    (v11)edge node[scale=2.3]{\text{water pump defect}}(v16)
    (v12)edge node[scale=2.2]{\text{corrosion/coating}}(v17)
    (v12)edge node[scale=2.5]{\text{loose fixing}}(v18)
    (v13)edge node[scale=2.5]{\text{fail}}(v19)
    (v13)edge node[scale=2.5]{\text{not fail}}(v20)
    (v14)edge node[scale=2.5]{\text{corrosion/coating}}(v21)
    (v14)edge node[scale=2.5]{\text{clamp}}(v22)
    
    (v15)edge node[scale=2.5]{\text{fail}}(r1)
    (v15)edge node[scale=2.5]{\text{not fail}}(r2)
    (v16)edge node[scale=2.5]{\text{fail}}(r3)
    (v16)edge node[scale=2.5]{\text{not fail}}(r4)
    
    (v17)edge node[scale=2.5]{\text{oil leak}}(v23)
    (v17)edge node[scale=2.2]{\text{pollution/coating}}(v24)
    (v17)edge node[scale=2]{\text{porcelain glazing defect}}(v25)
    (v18)edge node[scale=2.5]{\text{oil leak}}(v26)
    (v18)edge node[scale=2.2]{\text{breather defect}}(v27)
    (v18)edge[above] node[scale=2.2]{\text{arcing/corona}}(v28)
    (v18)edge[below] node[scale=2.5]{\text{ring defect}}(v28)
    (v21)edge node[scale=2.5]{\text{cracked}}(v29)
    (v21)edge node[scale=2.5]{\text{high resistance}}(v30)
    (v22)edge node[scale=2.5]{\text{clamp defect}}(v31)
    (v22)edge[above] node[scale=2.2]{\text{high RFI/}}(v32)
    (v22)edge[below] node[scale=2.2]{\text{Thermovision}}(v32)
    (v23)edge node[scale=2.5]{\text{fail}}(h1)
    (v23)edge node[scale=2.5]{\text{not fail}}(h2)
    (v24)edge node[scale=2.5]{\text{fail}}(h3)
    (v24)edge node[scale=2.5]{\text{not fail}}(h4)
     (v25)edge node[scale=2.5]{\text{fail}}(h5)
    (v25)edge node[scale=2.5]{\text{not fail}}(h6)
     (v26)edge node[scale=2.5]{\text{fail}}(h7)
    (v26)edge node[scale=2.5]{\text{not fail}}(h8)
     (v27)edge node[scale=2.5]{\text{fail}}(h9)
    (v27)edge node[scale=2.5]{\text{not fail}}(h10)
     (v28)edge node[scale=2.5]{\text{fail}}(h11)
    (v28)edge node[scale=2.5]{\text{not fail}}(h12)
    (v29)edge node[scale=2.5]{\text{fail}}(h13)
    (v29)edge node[scale=2.5]{\text{not fail}}(h14)
    (v30)edge node[scale=2.5]{\text{fail}}(h15)
    (v30)edge node[scale=2.5]{\text{not fail}}(h16)
    (v31)edge node[scale=2.5]{\text{fail}}(h17)
    (v31)edge node[scale=2.5]{\text{not fail}}(h18)
    (v32)edge node[scale=2.5]{\text{fail}}(h19)
    (v32)edge node[scale=2.5]{\text{not fail}}(h20)
    (v66)edge node[scale=2.5]{\text{fail}}(s9)
    (v66)edge node[scale=2.5]{\text{not fail}}(s10)
    (v4)edge node[scale=2.5]{\text{leak}}(v66)
     }}
    
    \caption{An example of a staged tree for the system of bushing.}
    \label{fig:event tree}
\end{figure}

A causal CEG orders events along its root-to-sink paths to be consistent with any hypothesised temporal ordering of events. This is especially useful when analysing the effects of a cause which through these semantics will appear downstream of the position on CEG that represents the cause. 

Our first step in customising the CEG framework to system reliability is to make a simple extension of the original definition by splitting the sink node $w_{\infty}$ into two terminal nodes. 
\begin{definition} A \textit{failure CEG} (FCEG) is a CEG with two \textit{terminal nodes} $w_{\infty}^f$ and $w_{\infty}^n$ so that the root-to-leaf path terminates at $w_{\infty}^f$ models the failure process of a system while the path terminates at $w_{\infty}^n$ expresses that the system has not yet failed. 
\end{definition}
In the rest of the paper, all the CEGs we refer to are FCEGs.
\begin{example}\label{fceg example}
The staged tree in Figure \ref{fig:event tree} describes some of the failure modes of a bushing, which is a component of a transformer. Suppose this represents our hypothesised probability model based on both expert judgements and available data. A CEG can be automatically constructed from this tree and is shown in Figure \ref{fig:ceg}. From this CEG we can conclude that the failure event of being worn out is independent of the components and symptoms for this system. The failure caused by the symptoms represented on $v_4$ and $v_{18}$ is independent of other events represented before these symptoms. Since $v_{28}$ and $v_{30}$ are at the same stage and in the same position $w_{20}$, we can read from this graph that a cracked component or a clamp defect are equally likely explanations of a failure event. 
\end{example}

\begin{figure}
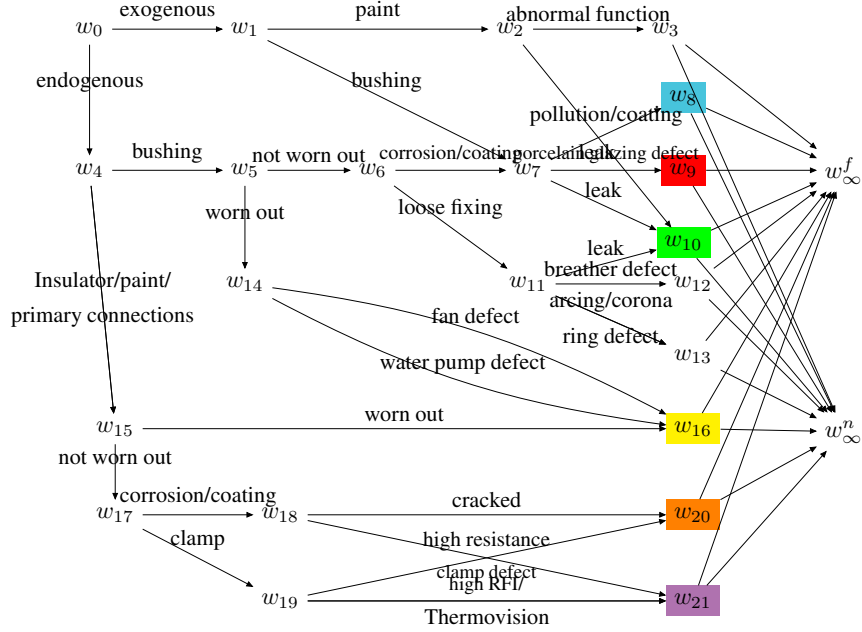

    \centering
    \resizebox{0.7\linewidth}{!}{
    \tikz{
    \node[scale=2.5](w0){$w_0$};
    \node[right = 4cm of w0,scale=2.5](w1){$w_1$};
    \node[right = 8cm of w1,scale=2.5](w2){$w_2$};
    \node[right = 4cm of w2,scale=2.5](w3){$w_3$};
    \node[below = 4cm of w0,scale=2.5](w5){$w_4$};
    \node[right = 4cm of w5,scale=2.5](w6){$w_5$};
    \node[right = 3cm of w6,scale=2.5](w7){$w_6$};
    \node[right = 4cm of w7,scale=2.5](w8){$w_7$};
    \node[right = 4cm of w8,scale=2.5,fill=red](w10){$w_{9}$};
    \node[above = 1.5cm of w10,scale=2.5,fill=SkyBlue](w9){$w_8$};
    \node[below = 1.5cm of w10,scale=2.5,fill=green](w11){$w_{10}$};
    
    \node[below = 3cm of w8,scale=2.5](w12){$w_{11}$};
    \node[right = 4cm of w12,scale=2.5](w13){$w_{12}$};
    \node[below = 1.5cm of w13,scale=2.5](w14){$w_{13}$};
    \node[below = 1.5cm of w14,scale=2.5,fill=yellow](w17){$w_{16}$};
    \node[left = 19cm of w17,scale=2.5](w16){$w_{15}$};
    \node[below = 3cm of w6,scale=2.5](w15){$w_{14}$};
    
    \node[below = 2cm of w16,scale=2.5](w18){$w_{17}$};
    \node[right = 4cm of w18,scale=2.5](w19){$w_{18}$};
    \node[below = 2cm of w19,scale=2.5](w20){$w_{19}$};
    \node[right = 13cm of w19,scale=2.5,fill=orange](w21){$w_{20}$};
    \node[below = 2cm of w21,scale=2.5,fill=Orchid](w22){$w_{21}$};
    \node[right = 4cm of w10,scale=2.5](wf){$w_{\infty}^f$};
    \node[below = 8cm of wf,scale=2.5](wn){$w_{\infty}^n$};
    
     \path[->,draw]
   (w0)edge[above] node[scale=2.3]{\text{exogenous}}(w1)
   (w1)edge[above] node[scale=2.3]{\text{paint}}(w2)
   (w2)edge[above] node[scale=2.3]{\text{abnormal function}}(w3)
   (w2)edge[below] node[scale=2.3]{\text{leak}}(w11)
   (w0)edge[above] node[scale=2.3]{\text{endogenous}}(w5)
   (w5)edge[above] node[scale=2.3]{\text{bushing}}(w6)
   (w6)edge[above] node[scale=2.3]{\text{not worn out}}(w7)
   (w7)edge[above] node[scale=2.1]{\text{corrosion/coating}}(w8)
   (w8)edge[above] node[scale=2.3]{\text{pollution/coating}}(w9)
   (w8)edge[above] node[scale=2]{\text{porcelain glazing defect}}(w10)
   (w8)edge[above] node[scale=2.3]{\text{leak}}(w11)
   (w1)edge[above] node[scale=2.3]{\text{bushing}}(w8)
   (w7)edge[above] node[scale=2.3]{\text{loose fixing}}(w12)
   (w12)edge[above] node[scale=2.3]{\text{leak}}(w11)
   (w12)edge[above] node[scale=2.3]{\text{breather defect}}(w13)
   (w12)edge[above] node[scale=2.3]{\text{arcing/corona}}(w14)
   (w12)edge[below] node[scale=2.3]{\text{ring defect}}(w14)
   (w5)edge[above] node[scale=2.3]{\text{Insulator/paint/}}(w16)
   (w5)edge[below] node[scale=2.3]{\text{primary connections}}(w16)
   (w16)edge[above] node[scale=2.3]{\text{worn out}}(w17)
   (w6)edge[above] node[scale=2.3]{\text{worn out}}(w15)
   (w15)edge[above,bend left=10] node[scale=2.3]{\text{fan defect}}(w17)
   (w15)edge[above,bend right=10] node[scale=2.3]{\text{water pump defect}}(w17)
   (w16)edge[above] node[scale=2.3]{\text{not worn out}}(w18)
   (w18)edge[above] node[scale=2.3]{\text{corrosion/coating}}(w19)
   (w18)edge[above] node[scale=2.3]{\text{clamp}}(w20)
   (w19)edge[above] node[scale=2.3]{\text{cracked}}(w21)
   (w19)edge[above] node[scale=2.2]{\text{high resistance}}(w22)
   (w20)edge[below] node[scale=2]{\text{clamp defect}}(w21)
   (w20)edge[above] node[scale=2]{\text{high RFI/}}(w22)
   (w20)edge[below] node[scale=2.3]{\text{Thermovision}}(w22)
   (w3)edge node[scale=1.5]{}(wf)
   (w9)edge node[scale=1.5]{}(wf)
   (w10)edge node[scale=1.5]{}(wf)
   (w11)edge node[scale=1.5]{}(wf)
   (w13)edge node[scale=1.5]{}(wf)
   (w14)edge node[scale=1.5]{}(wf)
   (w17)edge node[scale=1.5]{}(wf)
   (w21)edge node[scale=1.5]{}(wf)
   (w22)edge node[scale=1.5]{}(wf)
   (w3)edge node[scale=1.5]{}(wn)
   (w9)edge node[scale=1.5]{}(wn)
   (w10)edge node[scale=1.5]{}(wn)
   (w11)edge node[scale=1.5]{}(wn)
   (w13)edge node[scale=1.5]{}(wn)
   (w14)edge node[scale=1.5]{}(wn)
   (w17)edge node[scale=1.5]{}(wn)
   (w21)edge node[scale=1.5]{}(wn)
   (w22)edge node[scale=1.5]{}(wn)
    }}
    
    \caption{A CEG derived from the staged tree in Figure \ref{fig:event tree}.}
    \label{fig:ceg}
\end{figure}

\section{The mapping of a document on to a CEG}\label{sec:nlp} In order to embed causal objects within a CEG, we design a framework that maps texts in natural language on to paths of a CEG. Let $D$ represent all documents we pick for a specific context after basic text cleaning, for example the faults related to a specific system. 
 Each document consists of a sequence of $N_s$ words $\bm{s}=\{s_{1},...,s_{N_s}\}$, and $\bm{s}\in D$. We next define a function $J:(\bm{s},\Omega)\mapsto\lambda$. This enables us to map texts from each document on to a root-to-sink path $\lambda\in\Lambda(C)$ depicted on the CEG of the hypothesised process. Below we carefully define the required parameters set $\Omega$ and show in detail how to construct $J$.

We begin with preprocessing the documents. As demonstrated in Figure \ref{fig:framework}, this proceeds along two steps: (1) constructing a map to extract the causally ordered events from texts; (2) embedding the dependency between the extracted events within a causal network. Firstly we propose a new map for events extraction. This is inspired by Mirza and Tonelli \citep{Mirza2016} who exploited the interaction between the causal and the temporal dimensions of texts. Our approach combines two existing methods: the CAscading EVent Ordering architecture (CAEVO) \citep{Chambers2014} and the hierarchical causality network designed by Zhao \textit{et.al} \citep{causalnet}. The former extracts temporally ordered events whilst the latter extracts causally related events. Our contribution here is to add the temporally ordered events extracted from CAEVO to causally related events. This enables us to refine the sequences of the extracted events in a way that is compatible with a CEG representation. Note that by adopting the first work we have therefore implicitly made the common assumption - obviously applicable within this context -  a cause must happen before its effect.

We firstly specify grammar rules $\Omega_r$ to parse a sentence into phrases. Then we customise a set of \textit{shallow causal patterns} $\Omega_{s}$. This is defined as a set of tuples $\{(c,<A,B>,r_c(A,B))\}$, where $c$ denotes causal connectives which are trigger words of the \textit{shallow causal dependencies}, $<A,B>$ denotes the pair of events in this sentence that are connected by $c$ and $r_c(A,B)$ the relationship between $A$ and $B$ given $c$. 

Given $ \Omega_{NL}=\{\Omega_r,\Omega_{s}\}$, we can define a function $\sigma$ to extract the ordered events from each document $\bm{s}$. In this paper we call such events \textit{core events}. Denote these $n_u$ events by $\bm{u}=\{u_1,...,u_{n_u}\}$ and the ordering of $\bm{u}$ by $\pi_{\bm{u}}$. Let $\bm{U}_D$ and $\bm{\Pi}_D$ represent respectively the set of all core events and all possible orderings over them for dataset $D$. Then the codomian of $\sigma$ is $\bm{U}_D\times\bm{\Pi}_{D}$. We write this function as:
$$\sigma:(\bm{s},\Omega_{NL})\mapsto (\bm{u},\pi_{\bm{u}}).$$  The function $\sigma$ can be decomposed into a sequence of functions as explained below.
\subsection{The $\sigma$ function}\label{sequence of functions}
We first define a function $\alpha:(\bm{s},\Omega_r)\mapsto\bm{s}_{\bm{A}}$ that inputs a document $\bm{s}\in D$ and the known set of rules $\Omega_r$ that are used to parse a sentence into phrases and returns a set of $n$ phrases $\bm{s}_{\bm{A}} = \{\bm{s}_{A_1},...,\bm{s}_{A_n}\}$. A phrase $\bm{s}_{A_j}$ is a subvector of $\bm{s}$ whose subscript $A_j$ is a set of word indices $\{a_j,a_j+1,...,a_j+m_j\}\subseteq \{1,...,N_s\}$. Let $\bm{A}=\{A_1,...,A_n\}$. Note that $A_i\cap A_j=\emptyset$ for any $A_i,A_j\in \bm{A}$.

 Following the methods introduced by Zhao \textit{et.al} \citep{causalnet}, we next extract causality mentions $\bm{s}_{\bm{B}} = \{(\bm{s}_{B_1},\bm{s}_{B_2}),...,(\bm{s}_{B_{m-1}},\bm{s}_{B_{m}})\}$, where $\bm{B}=\{(B_1,B_2),...,(B_{m-1},B_{m})\}$. Each pair of causality mentions is a pair of cause-effect events extracted using the customised linguistic patterns $\Omega_s$. If more than one pattern is satisfied, then we output multiple pairs of causality mentions through the map $\beta:(\bm{s},\Omega_{s})\mapsto\bm{s}_{\bm{B}}$. Each causality mention $\bm{s}_{B_j}$ is a subvector of $\bm{s}$ with word indices $B_j=\{b_j,b_{j}+1,...,b_{j}+k_j\}\subseteq \{1,...,N_s\}$. 

We next introduce a new map $\gamma: (\bm{s}_{\bm{A}},\bm{s}_{\bm{B}})\mapsto\bm{s}_{\bm{C}}$ that combines the results from the first two functions and outputs causal phrases $\bm{s}_{\bm{C}} = \{(\bm{s}_{C_1},\bm{s}_{C_2}),...,(\bm{s}_{C_w-1},\bm{s}_{C_{w}})\}$, where  $\bm{C}=\{(C_1,C_2),...,(C_{w-1},C_{w})\}$. By this step, for causality mentions with index sets $(B_j,B_{j+1})$, if $\bm{s}_{B_j}$ consists of phrases in $\bm{s}_{\bm{A}}$, so that there exist $\{A_{j^1},...,A_{j^m}\}\subseteq B_j$ then we collect causal phrases with index sets $\{(C_{l^1},C_{l^2})...,(C_{l^{2m-1}},C_{j^{2m}})\}=\{(A_{j^1},B_{j+1}),...,(A_{j^m},B_{j+1})\}$. If there exists no $A_k\in \bm{A}$ such that $A_k\subseteq B_j$, then $(C_{l^1},C_{l^2})=(B_j,B_{j+1})$.

We can now deploy the idea of abstract causal events first proposed by Zhao \textit{et.al} \citep{causalnet}. Here the nouns and verbs in each specific causal phrase $\bm{s}_{C_j}$ are replaced by their hypernyms in WordNet (WN), denoted by $WN_D$, and classes in VerbNet (VN), denoted by $VN_D$, respectively. The other uninformative words are removed. By doing this we obtain a set of abstract causal events $\bm{v}_{\bm{K}}=\{\bm{v}_{K_1},...,\bm{v}_{K_w}\}$, where $\bm{K}=\{K_1,...,K_w\}$ and $K_j=\{k_{j},...,k_{j}+l\}\subseteq C_j$. The word in the original document $s_{k_j}$ is replaced by $v_{k_j}\in WN_D\cup VN_D$. Let $\pi_{\bm{K}}$ denote the order over $\bm{K}$. Then $(\bm{v}_{\bm{K}},\pi_{\bm{K}}) = \{(\bm{v}_{K_1},\bm{v}_{K_2}),...,(\bm{v}_{K_{w-1}},\bm{v}_{K_{w}})\}$. We represent this procedure by the map $\iota:\bm{s}_{\bm{C}}\mapsto (\bm{v}_{\bm{K}},\pi_{\bm{K}})$.


We next apply an existing programme called CAEVO \citep{Chambers2014} to extract a temporal ordering of events. Denote the set of the extracted events by $\bm{v}_{\bm{\xi}}=\{v_{\xi_1},...,v_{\xi_{n_V}}\}$ and the order of the word indices $\bm{\xi}$ by $\pi_{\bm{\xi}}$. In CAEVO, events usually refer to verbs. So $v_{\xi_j}\in VN_D$ for $j\in\{1,...,n_V\}$. We represent the operations implemented by CAEVO by a function
 $\mu:\bm{s}\mapsto (\bm{v}_{\bm{\xi}},\pi_{\bm{\xi}})$.


To refine the sequences of the ordered events, we now combine results from $\iota$ and $\mu$ by a function
\begin{equation}
\phi:(\bm{v}_{\bm{K}},\pi_{\bm{K}},\bm{v}_{\bm{\xi}},\pi_{\bm{\xi}})\mapsto(\bm{u},\pi_{\bm{u}}).
    \end{equation} 
    When $\pi_{\bm{\xi}}$ and $\pi_{\bm{K}}$ are consistent, then $\pi_{\bm{u}}=\pi_{\bm{K}}$ and $\bm{u}=\bm{v}_{\bm{K}}$. If there exists a partial ordering $\pi'\in\pi_{\bm{\xi}}$ of events $\bm{v}'\subseteq \bm{v}_{\bm{\xi}}$, where $\pi'\notin\pi_{\bm{K}}$ and this is consistent with $\pi_{\bm{K}}$, then we add $\pi'$ to $\pi_{\bm{K}}$: $\pi_{\bm{K}}\leftarrow (\pi_{\bm{K}},\pi')$. For any $v_i\in \bm{v}'$, if $v_{i}\notin \bm{v}_{K}$, we add this event to $\bm{v}_{\bm{K}}$: $\bm{v}_{\bm{K}}\leftarrow(\bm{v}_{\bm{K}},v_{i})$. When $\pi'$ contradicts the causal ordering $\pi_{\bm{K}}$, then $\pi'$ is ignored and $\pi_{\bm{K}}$ remains unchanged. This is because $\pi'$ is a temporal ordering which may not imply the events are casually related. The outcome of $\phi$ is $(\bm{u},\pi_{\bm{u}})=(\bm{v}_{\bm{K}},\pi_{\bm{K}})$.


We can now write $\sigma$ as a composition of the functions listed above by
\begin{equation}
    \sigma(\bm{s},\Omega_{NL})=\phi(\iota(\gamma(\alpha(\bm{s},\Omega_r),\beta(\bm{s},\Omega_{s}))),\mu(\bm{s})).
\end{equation}

 It is straightforward to check (see supplementary material \citep{Yusupp2021}) that the $\sigma$ function is well defined. An illustration of the construction of $\sigma$ is given in the supplementary material \citep{Yusupp2021}. The technical development we outline above includes only the essential element we need to describe the map from the texts onto our graphical models. Further details of these maps, their motivations and coding can be found in \citep{Yu20211}.

\subsection{The Global Net}

Having extracted the cause-effect pairs of core events through $\sigma$, we need to cluster events and the set of partial ordering consistently with the topology of a corresponding CEG, so as not to lose any extracted causal relationships. Therefore we need to add an extra but necessary pre-processing step that registers an implicit partial order on all the core events.

In order to register the order of the core events, we here define a new graphical framework called a \textit{Global Net} (GN) whose topology is a DAG $G^*=(V^*,E^*)$. Each vertex $v\in V^*$ corresponds to a \textit{core event variable} defined in the following way. The core event variables are constructed by clustering the extracted core events $\bm{u}$. These may need to import expert judgements or assumptions to supplement the extracted causal pairs. For example, the three events: ``red phase transformer'', ``yellow phase transformer'', and ``blue phase transformer'', can be taken as three states of a categorical variable $U_i$ that represents the phase of the machine. For the core event ``oil leak'', we create a binary variable to indicate whether oil leak has been observed. The methods used for this are rather domain specific and so we have relegated a description of how this step might be performed to the supplementary material \citep{Yusupp2021}. We call these constructed variables the core event variables and denote them by $\bm{U}=\{U_1,...,U_{n_U}\}$. There is an edge $e_{v_i,v_j}\in E^*$ from $v_i\in V^*$ to $v_j\in V^*$ whenever $v_i$ is a genuine cause \citep{Pearl2000} of $v_j$.

Here we learn the topology of the GN by extracting a BN with constraints on potential causal relations extracted from $(\bm{u},\pi_{\bm{u}})$ and using established methodologies to pick out the genuine causal relationships that we can read from the BN \citep{Pearl2000}. Because the focus of this paper is the causal algebra, for simplicity we adopt the simplest, most transparent and familiar method, the BN, for performing this registration step. 
Let $G=(V,E)$ denote the topology of the BN and $\mathcal{P}^*$ denote the joint probability distribution over $\bm{U}$ defined on $G$, where $V=V^*$ correspond to the core event variables and the edges set $E$ can be learned using any structure learning algorithm \citep{scutari2012}. However, the BN selected by the algorithm alone cannot ensure that every edge represents a putative causal dependency. Therefore, given the extracted cause-effect paired core events, we summarise the likely causal relationships between the constructed variables and accordingly create a list of directed edges that must appear in the BN. We can also create a list of directed edges that should never be present in the BN, which only include those violate the cause-before-effect temporal relations. This step allows us to preserve the putative cause-effect structure extracted from $\sigma$. In this way we attempt to filter through genuine causal relationships \citep{Pearl2000} and distinguish these from relationships that could not properly be hypothesised as causal. To illustrate this process, Figure \ref{fig:global net} (a) gives a BN extacted from the dataset of a subsystem of bushing. This has then been transformed to the GN in Figure \ref{fig:global net} (b) by the method we describe above. More details of this construction are given in the supplementary material \citep{Yusupp2021}.

\begin{figure}
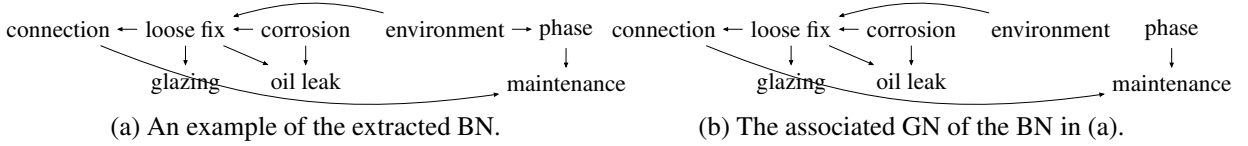

    \centering
   \resizebox{\linewidth}{!}{
    \tikz{
    \node[scale=3.5](r1){environment};
    \node[left = 1cm of r1,scale=3.5](r2){corrosion};
    \node[left = 1cm of r2,scale=3.5](r3){loose fix};
    \node[below = 1cm of r2,scale=3.5](r4){oil leak};
    \node[below = 1cm of r3,scale=3.5](r5){glazing};
    \node[left = 1cm of r3,scale=3.5](r6){connection};
    \node[below = 0.5cm of r6,scale=3.5](r7){};
    \node[right = 1cm of r1,scale=3.5](r10){phase};
    \node[below = 1cm of r10,scale=3.5](r11){maintenance};
    \node[below = 0.5cm of r4,scale=4](r12){\text{(a) An example of the extracted BN.}};
    
    \node[right = 20cm of r10,scale=3.5](a1){environment};
    \node[left = 1cm of a1,scale=3.5](a2){corrosion};
    \node[left = 1cm of a2,scale=3.5](a3){loose fix};
    \node[below = 1cm of a2,scale=3.5](a4){oil leak};
    \node[below = 1cm of a3,scale=3.5](a5){glazing};
    \node[left = 1cm of a3,scale=3.5](a6){connection};
    \node[below = 1cm of a6,scale=3.5](a7){};
    \node[right = 1cm of a1,scale=3.5](a10){phase};
    \node[below = 1cm of a10,scale=3.5](a11){maintenance};
   \node[below = 0.5cm of a4,scale=4](a12){\text{(b) The associated GN of the BN in (a).}};
    \path[->,draw]
    (r1)edge[bend right = 15] node{}(r3)
    (r2)edge node{}(r3)
    (r2)edge node{}(r4)
    (r3)edge node{}(r4)
    (r3)edge node{}(r5)
    (r3)edge node{}(r6)
    (r6) edge[bend right = 15] node{}(r11)
    (r1)edge node{}(r10)
    (r10)edge node{}(r11)
    
    (a1)edge[bend right = 15] node{}(a3)
    (a2)edge node{}(a3)
    (a2)edge node{}(a4)
    (a3)edge node{}(a4)
    (a3)edge node{}(a5)
    (a3)edge node{}(a6)
    (a6) edge[bend right = 15] node{}(a11)
    (a10)edge node{}(a11)
    
    }}
    \caption{The extracted BN and the GN for a bushing system.}
    \label{fig:global net}
\end{figure}
\subsection{The mapping of a document on to the GN} 
Given a document $\bm{s}\in D$, we can extract a cluster of ordered core events using the aforementioned method. These can then be identified on the GN through finding the associated vertices $V\subset V^*$ which lie on the subgraph $G=(V,E)$, where $E\subset E^*$.

Let $\psi:(\bm{u},\pi_{\bm{u}})\mapsto G$ be the function that matches $(\bm{u},\pi_{\bm{u}})$ to a subgraph $G$ and $L:(\bm{s},\Omega_{NL})\mapsto G$ be the function mapping a document on to a cluster on the GN, where $L\subseteq D\times GN$, so that $L = \psi\circ\sigma$. 
\begin{proposition}\label{map1}
When there are no core event variables that are unobservable, the mapping of a document on to the GN as defined above is unique and well-defined.
\end{proposition} 
The proof of this proposition is shown in the supplementary material \citep{Yusupp2021}.

\subsection{The GN-CEG causal model}\label{sec:hca}

Though causal semantics represented on the GN embed the ``shallow'' causal relations, some more complex relations between events raised in reliability cannot be well-modelled by a GN alone. For example, the asymmetry of the relationships of the variables $\bm{U}$ defined above cannot be embodied simply through fitting a BN. When an ``AND'' gate exists between two variables, the routine model selection of BNs alone cannot automatically extract such relationship unless we create a new variable and add it to the vertex set instead. However, these problems and other more subtle problems can be simply addressed by introducing a CEG at a deeper layer. We then design algorithms to map the GN on to a CEG so that we can learn the more general and concise causal patterns on the CEG. 

We therefore propose a two-layer hierarchical model, with a GN lying at the surface layer and a causal CEG lying at the deeper layer. We call this a GN-CEG causal model. Our first step is to embed the shallow causal dependency on to the GN using the function $L$ as demonstrated in the previous section. This gives us a subgraph $G$ of the GN. We then define a map to project $G$ on to a path $\lambda\in\Lambda(C)$ in the CEG. This projection can be well-defined since for each position along $\lambda$, represented as $w\in \lambda$, we can define a variable over its floret $F(w)$, which can be treated as a latent variable of a sequence of observations, \textit{i.e.} the core events whose corresponding vertices lie in $G$. We explain this in detail in the next section.

 Within these semantics this setup makes sense in that the text label of each edge on the CEG describes a single event which is triggered when a sequence of core events have been observed. We call such an event a \textit{d-event}. Let $X$ denote the set of all d-events labelled on $C$. For each d-event $x\in X$, we use $e(x)\subseteq E_C$ to represent the set of edges such that once $x$ is observed, the transition along any edge $e\in e(x)$ can happen. Similarly, we use $w(x)\subseteq V_C$ to represent the set of receiving vertices of edges in $e(x)$. Every d-event $x$ can also be interpreted by a set of core events that are represented on the GN.
Therefore, to each d-event $x$, we can identify a subgraph on the GN containing this set of core event variables associated with $e(x)$ on the CEG.

The construct of a d-event forms the basis for clustering the core event variables and matching each cluster to a latent position. However, usually we may need to design algorithms to automatically learn these clusters and their assignments to the latent variables by estimating the associated conditional probabilities. Since the focus of this paper is on structural modelling we will assume that parameter estimation has already happened. For example to estimate the transition probabilities $\bm{\theta}_w$ for every $ w\in V_{C}$ given observations from a Bayesian analysis, we could assign Dirichlet priors to these probabilities. Let $\Omega_G$ denote the posterior expectation of such parameters. Then given $G$ and $\Omega_G$ we can learn the latent position for every core event, and a latent path for $G$. We represent this procedure by a map $Q:(G,\Omega_G)\mapsto\lambda$. Through projecting core events at the causal CEG, deeper causal dependency can be identified. This is encoded within the more refined path based events represented on the CEG. Let $\Omega = \{\Omega_{NL},\Omega_G\}$, the function to map a document to a latent path is then $J=Q\circ L$.

\section{Nested causality}\label{sec:rc}
 The GN-CEG causal model is analogous to the \textit{recursive Bayes Net} (RBN) \citep{Williamson2005} that models nested causal relationships. So here we define the terminologies for our model in light of the RBN. 

The RBN is a BN over $N$ variables $V=\{V_1,...,V_N\}$ where some variables can take BNs as values. If $V_i\in V$ is a variable whose values index a set of BNs, denoted by $G(V_i)$, then this is a \textit{network variable}. The RBN contains at least one such variable. If a variable corresponds to a node in $G(V_i)$, then it is said to be a \textit{direct inferior} of $V_i$ and $V_i$ the \textit{direct superior} of it. 

 
 For $V_i\in V$, let $NID(V_i)$ denote the non-inferiors or descendents of $V_i$ and
  let $Dsup(V_i)$ denote the direct superiors of $V_i$. Within an RBN, the following conditional independence is assumed to be true.
 
 \begin{assumption}[Recursive Causal Markov Condition (RCMC)\citep{Williamson2005,Casini2011}]\label{theorem:rcmc} For an RBN, every variable $V_i\in V$ is independent of those variables that are neither its descendants nor its inferiors conditional on its parents and its direct superiors. So we have
  $$V_i\indep NID(V_i)|pa(V_i)\cup Dsup(V_i).$$

\end{assumption} 


We next define analogous concepts within our model. 
\subsection{Communities}\label{collocation} 
We can always define a measurable variable over each floret on a simple CEG \citep{Wilkerson2020}. 
Thus for $w\in V_C$, we define a \textit{floret variable} $Y(w)$ for each $F(w)$. Let $e_{w,w'}\in E(w)$ be an emanating edge of $w$. Let $e_{w,w'}\in\lambda$ represent that the edge $e_{w,w'}$ is along the path $\lambda$. Then given a path $\lambda\in\Lambda(C)$,
 \begin{equation}
     Y(w)=\begin{cases}
     &y_{w,w'},
\text{\hspace{4mm}if $e_{w,w'}\in \lambda$,}\\
&0, \text{\hspace{12mm}if $e_{w,w'}\notin \lambda$.}\end{cases}
 \end{equation} 
 
 For each $w\in V_C$, we can also define an \textit{incident variable} $I(w)$ to indicate whether a path $\lambda\in\Lambda(C)$ passes through $w$ \citep{Wilkerson2020} so that
\begin{equation}
    I(w)=\begin{cases}
    &1, \text{\hspace{4mm} if $w\in\lambda$,}\\
    &0, \text{\hspace{4mm} if $w\notin\lambda$.}
    \end{cases}
\end{equation}

As mentioned in the previous section, the innovation of this paper is to map a document to a latent path on the CEG through the GN-CEG model. So for each document, a sequence of core events $\bm{u}$ can be extracted, which can then be clustered into a sequence of subsets $\bm{u}_1,...\bm{u}_d$ with respect to the order given by the GN so that each subset of core events trigger a position on the CEG to be passed through, or equivalently an edge to be passed along. Here $\bm{u}_i\cap\bm{u}_j=\emptyset$ for $i,j\in\{1,...,d\}$. The sequence of the triggered edges $\{e_1,...,e_d\}$ lie on the same root-to-sink path. This motivates us to find a generic approach to cluster the core events given their causal order so that each cluster associates to a latent edge on the CEG. Though in this paper the causal order is registered on a GN, it is clear that this idea allows extension to trees or CEGs implicitly where core events can also be represented and their causal order can be embedded. 
Following this idea, when an edge is triggered, equivalently, the associated floret variable takes a particular value, a set of core events are supposed to be observed. So a set of core event variables correspond to a value of a floret variable, which is then analogously a network variable in light of the  RBN \citep{Casini2011,Williamson2005}. 

Each floret variable is the latent variable of a set of core event variables. Here we call this set of core event variables a \textit{community}. 
Let $\bm{U}_i\subseteq \bm{U}$ denote the community associated with $Y(w_i)$, which consists of core event variables interpreting the d-events represented on $F(w_i)$. The subgraph $G_i\subseteq G^*$ is defined for $\bm{U}_i$, where $G_i=(E_i,V_i)$. The vertex set $V_i$ includes the vertices corresponding to the core event variables in $\bm{U}_i$. The edge set $E_i$ is constructed such that for $w\in V_i$, if $e_{w',w}\in E^*$ where $w'\in V^*$, then $e_{w',w}\in E_i$. The value taken by $Y(w_i)$, say $y_{w_i,w_k}$, determines a sub-community $\bm{U}_{i,k}\subseteq \bm{U}_i$ and a BN over it, denoted by $G_{i,k}\subseteq G_i$. Under such setting, the GN-CEG model is a nested causal network analogous to a two-level RBN.

  
 If there exists a total order over the communities which is consistent with the partial order of the core events that has been extracted in the pre-processing step, then we can design an artificial timeline for the communities. When two communities $\bm{U}_i$ and $\bm{U}_j$, $i<j$, are ordered at the same position by the total order, then they are collocated along the timeline. But this does not mean that the events included in these two communities happen simultaneously since we have a registered partial order over them.

 If $\bm{U}_i$ and $\bm{U}_j$ overlap and $\bm{U}_{ij}=\bm{U}_i\cap \bm{U}_j$, then the florets $F(w_i)$ and $F(w_j)$ usually represent similar information, such as symptoms. In this case, we avoid aligning such florets on the same path so that there are no recurrent root causes, symptoms or defects represented along any failure or deteriorating path. Thus we make the following assumption. 
  
    \begin{assumption}\label{assumption: order}
   If two communities $\bm{U}_i$ and $\bm{U}_j$ overlap, then $\bm{U}_i$ and $\bm{U}_j$ are collocated along the timeline. Let $\bm{U}_{i\bar{j}} = \bm{U}_i/\bm{U}_{ij}$ and $\bm{U}_{\bar{i}{j}} = \bm{U}_j/\bm{U}_{ij}$. Then for any $U'\in\bm{U}_{i\bar{j}}$ and $U''\in\bm{U}_{\bar{i}{j}}$, represented on vertices $v'$ and $v''$ respectively, there is no edge connecting them: $e_{v',v''}\notin E^*$.
  \end{assumption}

 
 
 
 
 \begin{example}
 Given an example of a GN-CEG model with a BN in Figure \ref{fig:RCN}(a) and a CEG in Figure \ref{fig:RCN}(b), we have communities $\bm{U}_0=\{U_1,U_2,U_3,U_4\}$, $\bm{U}_1 = \{U_5,U_6\}$ and $\bm{U}_2=\{U_6,U_7\}$. Let $\bm{U}_{0,1}$ denote the sub-community corresponding to $Y(w_0)=y_{w_0,w_1}$ and suppose $\bm{U}_{0,1}=\{U_1,U_3,U_4\}\subseteq \bm{U}_0$. The causal relations between variables in $\bm{U}_{0,1}$ have been determined in the GN, so the subgraph $G_{0,1}$ associated with $y_{w_0,w_1}$ can be plotted as Figure \ref{fig:RCN}(c). 
 
 Suppose that this CEG portrays the following d-events: $x_{0,1}, x_{0,2}, x_{1,1}, x_{1,2}, x_{1,3}$. Let $w(x_{0,1})=w_1$, $w(x_{0,2})=w_2$, $w(x_{1,1})=w(x_{1,3})=w_{\infty}^f$, $w(x_{1,2})=w_{\infty}^n$. Then the floret $F(w_0)$ contains information about $x_{0,1}$ and $x_{0,2}$, the floret $F(w_1)$ contains information about $x_{1,1}$ and $x_{1,2}$, and the floret $F(w_2)$ contains information about $x_{1,2}$ and $x_{1,3}$. So we can also identify a subgraph on the GN for each d-event. For example, the subgraph associated with $x_{0,1}$ is the graph in (c).
 \end{example}

To complete this section, we first demonstrate some properties of the GN-CEG model before proceeding to develop an intervention calculus for it. Then we explain how to identify the effects when forcing a core event to happen.

 \begin{figure}
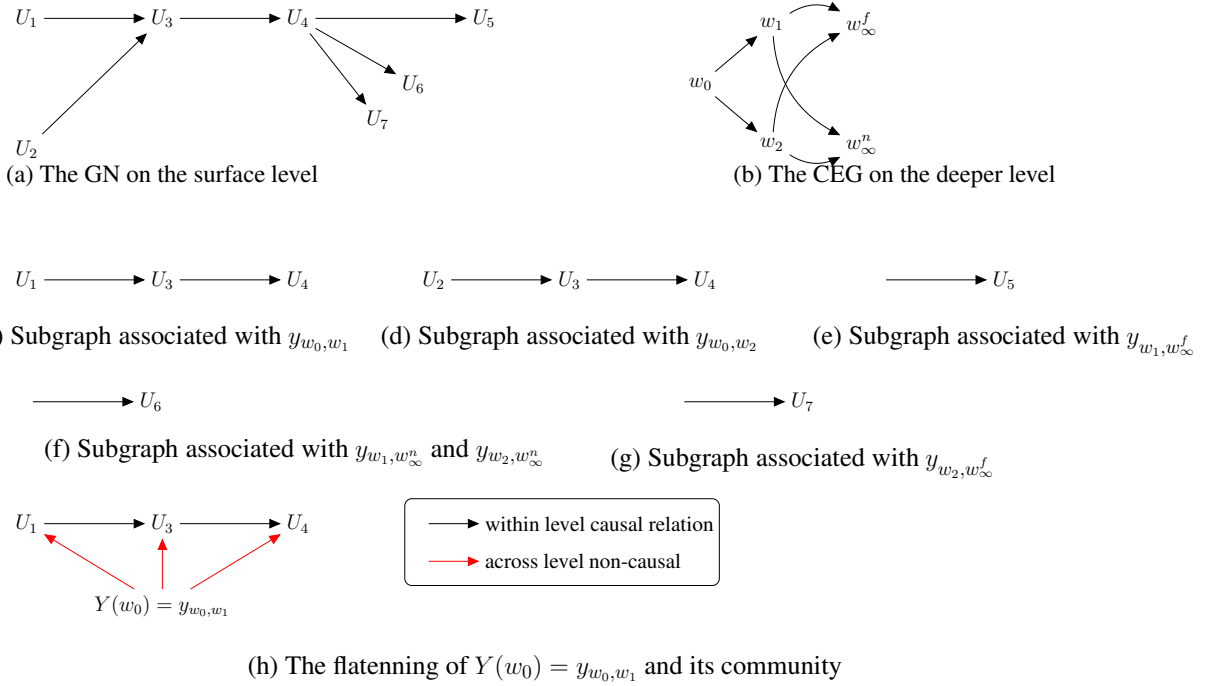

    \centering
     \resizebox{\linewidth}{!}{
    \tikz{
\node[font=\fontsize{12}{0}\selectfont](u1){$U_1$};
\node[below = 2cm of u1,font=\fontsize{12}{0}\selectfont](u2){$U_2$};
\node[right = 2cm of u1,font=\fontsize{12}{0}\selectfont](u3){$U_3$};
\node[right = 2cm of u3,font=\fontsize{12}{0}\selectfont](u4){$U_4$};
\node[right = 3cm of u4,font=\fontsize{12}{0}\selectfont](u5){$U_5$};
\node[below left = 1cm of u5,font=\fontsize{12}{0}\selectfont](u6){$U_6$};
\node[below left = 2cm of u5,font=\fontsize{12}{0}\selectfont](u7){$U_7$};
\node[below = 2.5cm of u3,font=\fontsize{14}{0}\selectfont](t1){\text{(a) The GN on the surface level}};
\node[right = 8cm of t1,font=\fontsize{14}{0}\selectfont](t2){\text{(b) The CEG on the deeper level}};
\node[right = 5cm of u6,font=\fontsize{12}{0}\selectfont](w1){$w_0$};
\node[above right = 1cm of w1,font=\fontsize{12}{0}\selectfont](w2){$w_1$};
    \node[below right = 1cm of w1,font=\fontsize{12}{0}\selectfont](w3){$w_2$};
    \node[right = 1cm of w2,font=\fontsize{12}{0}\selectfont](winf){$w_{\infty}^f$};
    \node[right = 1cm of w3,font=\fontsize{12}{0}\selectfont](winf2){$w_{\infty}^n$};
 
 \node[below = 2cm of u2,font=\fontsize{12}{0}\selectfont](u11){$U_1$};
 \node[right = 2cm of u11,font=\fontsize{12}{0}\selectfont](u12){$U_3$};
 \node[right = 2cm of u12,font=\fontsize{12}{0}\selectfont](u13){$U_4$};
 \node[below = 0.5cm of u12,font=\fontsize{15}{0}\selectfont](t3){\text{(c) Subgraph associated with $y_{w_0,w_1}$}};
 
 \node[right = 2cm of u13,font=\fontsize{12}{0}\selectfont](u21){$U_2$};
 \node[right = 2cm of u21,font=\fontsize{12}{0}\selectfont](u22){$U_3$};
 \node[right = 2cm of u22,font=\fontsize{12}{0}\selectfont](u23){$U_4$};
 \node[below = 0.5cm of u22,font=\fontsize{15}{0}\selectfont](t4){\text{(d) Subgraph associated with $y_{w_0,w_2}$}};
 
 \node[right = 3cm of u23](emp1){};
 \node[right = 2cm of emp1,font=\fontsize{12}{0}\selectfont](u31){$U_5$};
 
 \node[below = 2cm of u11](emp2){};
 \node[right = 2cm of emp2,font=\fontsize{12}{0}\selectfont](u32){$U_6$};
 \node[right = 10cm of u32](emp3){};
 \node[right = 2cm of emp3,font=\fontsize{12}{0}\selectfont](u33){$U_7$};
 \node[below = 0.5cm of u31,font=\fontsize{15}{0}\selectfont](t5){\text{(e) Subgraph associated with $y_{w_1,w_{\infty}^f}$}};
 \node[below right= 0.5cm and 0.1cm of emp2,font=\fontsize{15}{0}\selectfont](t6){\text{(f) Subgraph associated with $y_{w_1,w_{\infty}^n}$ and $y_{w_2,w_{\infty}^n}$}};
 \node[below = 0.5cm of u33,font=\fontsize{15}{0}\selectfont](t7){\text{(g) Subgraph associated with $y_{w_2,w_{\infty}^f}$}};
 
 \node[below = 2cm of emp2,font=\fontsize{12}{0}\selectfont](u41){$U_1$};
 \node[right = 2cm of u41,font=\fontsize{12}{0}\selectfont](u43){$U_3$};
 \node[right = 2cm of u43,font=\fontsize{12}{0}\selectfont](u44){$U_4$};
 \node[below = 1cm of u43,font=\fontsize{12}{0}\selectfont](y1){$Y(w_0)=y_{w_0,w_1}$};
   
\node[right = 2cm of u44](legend1){};
\node[right = 1cm of legend1,font=\fontsize{12}{0}\selectfont](legend2){\text{within level causal relation}};
\node[below = 0.5cm of legend1](legend3){};
\node[right = 1cm of legend3,font=\fontsize{12}{0}\selectfont](legend4){\text{across level non-causal}};
\node[below right = 0.5cm and 0.1cm of y1,font=\fontsize{15}{0}\selectfont](t8){\text{(h) The flatenning of $Y(w_0)=y_{w_0,w_1}$ and its community}};
    \path[->,draw]
    (w2)edge[bend left = 30](winf)
    (w2)edge[bend right = 30](winf2)
    (w3)edge[bend left = 30](winf)
    (w3)edge[bend right = 30](winf2)
    (w1)edge(w2)
    (w1)edge(w3)
    (u1)edge(u3)
    (u2)edge(u3)
    (u3)edge(u4)
    (u4)edge(u5)
    (u4)edge(u6)
    (u4)edge(u7)
    
    (u11)edge(u12)
    (u12)edge(u13)
    (u21)edge(u22)
    (u22)edge(u23)
    (emp1)edge(u31)
    (emp2)edge(u32)
    (emp3)edge(u33)
    (y1)edge[color=red](u41)
    (y1)edge[color=red](u43)
    (y1)edge[color=red](u44)
    (u41)edge(u43)
    (u43)edge(u44)

    (legend1)edge(legend2)
    (legend3)edge[color=red](legend4);

 \plate[inner sep=0.1cm,xshift=-0.12cm,yshift=0.12cm]{plate1}{(legend1)(legend2)(legend3)(legend4)}{};
    } }
    \caption{An example of the GN-CEG model whose surface BN is shown in (a) and the bottom CEG is shown in (b). The values of the floret variables and their corresponding subgraphs are depicted in (c)-(g).}
    \label{fig:RCN}
\end{figure}

\subsection{Properties of the GN-CEG}\label{properties of the gn-ceg}
 Given a core event variable $U_i$, we can identify the direct superior of $U_i$ on the CEG by the following steps. We have defined d-events for each $U_i$, so when observing a value of $U_i$ the corresponding d-events are observed. These d-events are labelled on the set of edges $E_{U_i}$ on the CEG. The set of receiving vertices of edges $e\in E_{U_i}$ can be treated as the latent states of $U_i$ on the CEG, denoted by $W_{U_i}$. Let  
\begin{equation}
    W_{i}=\bigcup_{w\in W_{U_i}}pa(w).
\end{equation} The set of florets whose root vertex is a $w\in W_{i}$ is denoted by $F(W_{i})$. Then the latent variables of $U_i$ can be represented by $Y(W_{i})=(Y(w))_{w\in W_{i}}$. 

Next we show the assumption of RCMC is valid within the GN-CEG model. We begin with proving the validity of the assumptions of the causal Markov condition (CMC) and the recursive Markov condition (RMC) \citep{Casini2011,Williamson2005}.

Smith and Anderson \citep{Smith2008} showed different conditional independence properties can be read from the CEG. This was recently generalised in \citep{Wilkerson2020}. Let $\tilde{V}=V_C/\{w_{\infty}^f,w_{\infty}^n\}$ denote the non-terminal nodes and $H$ denote an assignment of values to the variables defined over the CEG,
\begin{equation}
    H=\{y(w)\}_{w\in \tilde{V}}\cup\{i(w)\}_{w\in V_C}.
\end{equation}  

Let $Y_{nd}(w)$ denote the floret variables defined over the florets that do not lie downstream of $w\in V_C$. Given $H$, for every core event variable $U_i\in \bm{U}$, let $nd^H(U_i)$ denote the non-descendents of $U_i$, $pa^H(U_i)$ denote the parent variables of $U_i$ represented on the GN, and $Dsup^H(U_i)$ denote the direct superiors of $U_i$. 

\begin{proposition}
Given an assignment $H$, the causal Markov condition is valid within the GN-CEG model, whenever any variable that appears as a vertex is independent of its non-descendants given its parents. This implies in particular that:
\begin{equation}
    Y(w)\indep Y_{nd}(w)|I(w)
\end{equation}and
\begin{equation}\label{CMC}
    U_i\indep nd^H(U_i)|pa^H(U_i),Y^H(W_{i}).
\end{equation}
\end{proposition}

\begin{proof}
 Neither the floret variables nor the incident variables have direct superiors since they lie at the lowest level of the model. So we have the desired property \citep{Smith2008,Wilkerson2020}:
\begin{equation}
    Y(w)\indep Y_{nd}(w)|I(w).
\end{equation}We have shown that the latent variables of $U_i$ is $Y^H(W_{i})$, so
\begin{equation}
   Dsup^H(U_i) =Y^H(W_{i}).
\end{equation}
Let $\bm{U}'$ represent the communities associated to $Dsup^H(U_i)$ where $U_i\in \bm{U}'$. As mentioned in the beginning of this section, the value of $Dsup^H(U_i)$ determines $\bm{U}'$. Therefore we have
\begin{equation}\label{CMC}
    U_i\indep nd^H(U_i)|pa^H(U_i),Y^H(W_{i}).
\end{equation}
\end{proof}

Recall that a \textit{flattening} \citep{Casini2011,Williamson2005} is a non-recursive BN derived from an RBN which has the same domain as the RBN and contains arrows that do not correspond to causal relations. Here we can also construct a flattening for our nested causal network with respect to $H$, denoted by $H^{\downarrow}$. The vertex set of $H^{\downarrow}$ is the same as the vertex set of the nested causal network. There is an arrow from $V_i$ to $V_j$ in $H^{\downarrow}$ if $V_i$ is the parent or the direct superior of $V_j$ in the nested causal network. Let $pa^{H^{\downarrow}}(U_i)$ denote the set of parent variables of $U_i$ in $H^{\downarrow}$. Then 
\begin{equation}
    pa^{H^{\downarrow}}(U_i)=\{pa^H(U_i),Dsup^H(U_i)\}.
\end{equation} The dependence queries are not changed by ``flattening'' the nested causal network \citep{Williamson2005}. Therefore for every $U_i\in\bm{U}$,
\begin{equation}
    p(u_i|pa^{H^{\downarrow}}(u_i))=p(u_i|pa^H(u_i),y^H(W_{i})).
\end{equation}

Casini \textit{et.al} \citep{Casini2011} stated the recursive Markov condition for the RBN. This assumes that each variable is independent of the variables which are non-inferiors or peers given its direct superiors. To mirror this assumption on the GN-CEG model, let $Y(\overline{W_{i}})$ denote the set of floret variables defined over $\overline{W_{i}}=V_C/\{W_{i}\cup w_{\infty}^f\cup w_\infty^n\}$, then we have
\begin{equation}\label{RMC}
    U_i\indep (Y^H(\overline{W_{i}}),I^H(V_C))|Y^H(W_{i}),nd^H(U_i).
\end{equation}
\begin{theorem}
Given the CMC and the RMC to be true within the GN-CEG model, the Recursive Causal Markov Condition (RCMC) is valid:
\begin{equation}\label{RCMC}
    U_i\indep (nd^H(U_i),Y^H(\overline{W_{i}}),I^H(V_C))|pa^H(U_i),Y^H(W_{i}).
\end{equation}
\end{theorem}
\begin{proof}
When the CMC and the RMC are true, we have the conditional independence statements (\ref{CMC}) and (\ref{RMC}). By the contraction axiom of Pearl \citep{Pearl1997}, these imply (\ref{RCMC}).
\end{proof}

Within a GN, the core event variables are causally ordered. Within a CEG, the floret variables and the incident variables are causally ordered. Assume there exists a total order over the variables defined on the CEG, denoted by $\Pi$, which respects the causal ordering. 
For any assignment $H$, by assigning the values in $H$ one by one to the floret variables and the incident variables in an order consistent with $\Pi$, we can order the corresponding communities. As mentioned in Section \ref{collocation}, this order should be consistent with the partial order of the core event variables. Each value of a floret variable corresponds to a subgraph on the GN as defined above. When each pair of these non-recursive subgraphs is consistent under $H$, then $H$ is a \textit{consistent assignment} \citep{Casini2011,Williamson2005}. This concept is useful for proving the following proposition.

\begin{proposition}\label{prop:unique}
For the GN-CEG model, there is a unique map from the florets on the CEG to the core event variables on the GN.
\end{proposition}
To prove the uniqueness of the map, it is sufficient to prove the following two statements:(1) there exists a consistent assignment over the GN-CEG network; (2) the joint distribution is unique. The details of this proof is in the supplementary material \citep{Yusupp2021}.

\begin{proposition}\label{map2}
When there are no core event variables that are unobservable, the function $Q$ defined in Section \ref{sec:hca} is unique. As a result, the mapping from a document $\bm{s}\in D$ to $\lambda\in\Lambda(C)$ is unique. 
\end{proposition}
This is a straightforward consequence of the unique map. However it is tedious to confirm so we have relegated its proof to the supplementary material \citep{Yusupp2021}.
\subsection{The control of a core event}
Here we discuss an intervention regime on the GN, \textit{i.e.} forcing a single core event variable to take a specific value. The effect of this type of intervention on the GN and its corresponding factorisation is well established -- see for example Pearl \citep{Pearl2000}. We give a simple example below.

In the nested causal network shown in Figure \ref{fig:RCN}, suppose the core event $u_3$ is controlled to occur. The variable $U_3$ lies in the community $\bm{U}_0$ whose latent variable is $Y(w_0)$. The node $w_0$ has two child nodes $\{w_1,w_2\}$. The labels on $e_{w_0,w_1}$ and $e_{w_0,w_2}$ are $x_{0,1}$ and $x_{0,2}$ respectively.

When $Y(w_0)=y_{w_0,w_1}$, the d-event is $x_{0,1}$. The subgraph $G_{0,1}$ in this case is shown in Figure \ref{fig:RCN} (c). The only parent variable of $U_3$ is $U_1$. By Pearl's \textit{do}-algebra, given $do(U_3=u_3)$, the edge between $U_1$ and $U_3$ is removed. When $Y(w_0)=y_{w_0,w_2}$, the parent variable of $U_3$ is now $U_2$, see $G_{0,2}$ in Figure \ref{fig:RCN} (d). When forcing $U_3$ to take a value $u_3$, $U_2$ is no longer a parent of $U_3$. Figure \ref{fig:intervention a} depicts the topology of the post-intervention subgraphs. 
\begin{figure}
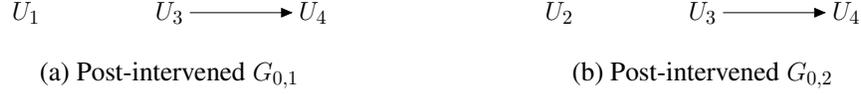

    \centering
     \resizebox{0.7\linewidth}{!}{
    \tikz{
\node[font=\fontsize{15}{0}\selectfont](u1){$U_1$};
\node[right = 2cm of u1,font=\fontsize{15}{0}\selectfont](u3){$U_3$};
\node[right = 2cm of u3,font=\fontsize{15}{0}\selectfont](u4){$U_4$};

\edge{u3}{u4};
\node[below = 0.5cm of u3,font=\fontsize{15}{0}\selectfont]{(a) Post-intervened $G_{0,1}$};
\node[right = 4cm of u4,font=\fontsize{15}{0}\selectfont](u2){$U_2$};
\node[right = 2cm of u2,font=\fontsize{15}{0}\selectfont](u23){$U_3$};
\node[right = 2cm of u23,font=\fontsize{15}{0}\selectfont](u24){$U_4$};
\edge{u23}{u24};
\node[below = 0.5cm of u23,font=\fontsize{15}{0}\selectfont]{(b) Post-intervened $G_{0,2}$};
    }
    }
    \caption{The manipulated subgraphs when intervening on $U_3=u_3$.}
    \label{fig:intervention a}
\end{figure}
The cross-level effect of this intervention on the CEG can be computed analogously to the way we calculate this effect in an RBN \citep{Williamson2005}. 
Suppose we are interested in the effect on the probability of failure, \textit{i.e.} the probability of arriving at $w^f_{\infty}$. Then this effect can be computed using the following formula:
\begin{equation} \label{control event}
        \pi(\Lambda_{w_{\infty}^f}||U_3=u_3)=\frac{\pi(\Lambda_{w_{\infty}^f},u_3)}{\pi(u_3)}.
\end{equation}This is analogous to the formula obtained using Pearl's causal algebra \citep{Pearl2000}. By combining the two values that $Y(w_0)$ can take, we have that
\begin{equation}
     \pi(\Lambda_{w_{\infty}^f},u_3)=\pi(\Lambda_{w_{\infty}^f},u_3,\Lambda_{w(x_{0,1})}) + \pi(\Lambda_{w_{\infty}^f},u_3,\Lambda_{w(x_{0,2})}).
\end{equation}
Applying the RCMC then:
\begin{equation}
    \begin{split}
      \pi(\Lambda_{w_{\infty}^f},u_3) = \pi&(\Lambda_{w_{\infty}^f}|\Lambda_{w(x_{0,1})})\pi(\Lambda_{w(x_{0,1})})\pi(u_3|\Lambda_{w(x_{0,1})})\\
        &+ \pi(\Lambda_{w_{\infty}^f}|\Lambda_{w(x_{0,2})})\pi(\Lambda_{w(x_{0,2})})\pi(u_3|\Lambda_{w(x_{0,2})})\\
       =\pi&(\Lambda(w(x_{0,1}),w_{\infty}^f)) \pi(u_3|\Lambda_{w(x_{0,1})})\\
       & + \pi(\Lambda(w(x_{0,2}),w_{\infty}^f))\pi(u_3|\Lambda_{w(x_{0,2})}).
    \end{split}
\end{equation}The dependence of $U_3$ on $U_1$ and $U_2$ is removed because of the intervention on $U_3$. We therefore have that:
\begin{equation}
   \pi(u_3) = \pi(u_3|\Lambda_{w(x_{0,1})})\pi(\Lambda_{w(x_{0,1})}) + \pi(u_3|\Lambda_{w(x_{0,2})})\pi(\Lambda_{w(x_{0,2})}).
\end{equation}
 The manipulated probability shown in equation \ref{control event} can therefore be expressed as:
\begin{equation}
     \pi(\Lambda_{w_{\infty}^f}||U_3=u_3)
    =\frac{\pi(\Lambda(w(x_{0,1}),w_{\infty}^f)) \pi(u_3|\Lambda_{w(x_{0,1})}) + \pi(\Lambda(w(x_{0,2}),w_{\infty}^f))\pi(u_3|\Lambda_{w(x_{0,2})})}{\pi(u_3|\Lambda_{w(x_{0,1})})\pi(\Lambda_{w(x_{0,1})}) + \pi(u_3|\Lambda_{w(x_{0,2})})\pi(\Lambda_{w(x_{0,2})})}.
\end{equation}

\section{Remedial interventions}\label{sec:intervention}


Having a causal CEG at the deeper layer, we now show how to compute causal effects of interventions within the GN-CEG model. In this paper we only focus on a special type of intervention called \textit{remedial intervention} which aims to prevent the same defect or failure reoccurring by exploring the root cause of a fault that has occurred and correcting it. 

One of the main applications which have to inspire standard approaches to causality are to health sciences where treatments are given and where intervention effects intend to mirror effects that might be observed within a randomised control trial \citep{Pearl2000,Rubin2003}. However, the term ``treatment'' is barely used in fault analysis where instead the concept of an intervention is one that provides a ``\textit{remedy}''. The inferential framework focuses on the discovery of a root cause of a fault and the identification of a sequence of actions that will provide a remedy to that fault. Here we develop a causal algebra where remedial maintenance takes centre stage. This new algebra is different from but analogous to Pearl's intervention calculus on BN \citep{Pearl2000}.

\subsection{Different types of remedy}\label{section:types of remedies}
 In the maintenance literature, there are three main categories of maintenance: \textit{perfect maintenance}, \textit{imperfect maintenance} and \textit{uncertain maintenance} \citep{Iung2005, Borgia2009, Cai2013}. Here, we develop this idea to explore the three types of remedy: the \textit{perfect remedy}, the \textit{imperfect remedy} and the \textit{uncertain remedy}. The classification of remedial actions is determined by the status of the defect part after its repair. The crucial factor deciding whether the status is as good as new (AGAN) \citep{Bedford2001} is whether the root cause of the defect is corrected. Note that here the status of the whole repairable system may not be AGAN since the other non-repaired components could have degraded. 

A remedial intervention is called a \textit{perfect remedy} if after the intervention the status of the remedied part returns to AGAN \citep{Andrzejczak2015, Bedford2001} and the root cause of the defect has been well identified and corrected. For example, if the failure of the bushing is only caused by a worn-out insulator, then replacing that insulator by a new one is a perfect remedy.

We can demonstrate the deteriorating process and the maintenance process for a perfect remedy on the graph shown in Figure \ref{fig:remedy}(a). The path shown in the figure is the only part of the root-to-sink path of an event tree or a CEG that models the deteriorating or failure process of the machine. Its root node represents the AGAN status of the component. The sink node here is assumed to represent the condition of the component just before the maintenance. After deploying a perfect remedy, by definition, the status will return to the root node of this sub-path from the sink node. Such recovery of the component is represented by the black dashed line. We call this a \textit{recovery path}. It delivers the effect of a remedy and returns the machine's status to full working order. 

\begin{figure}
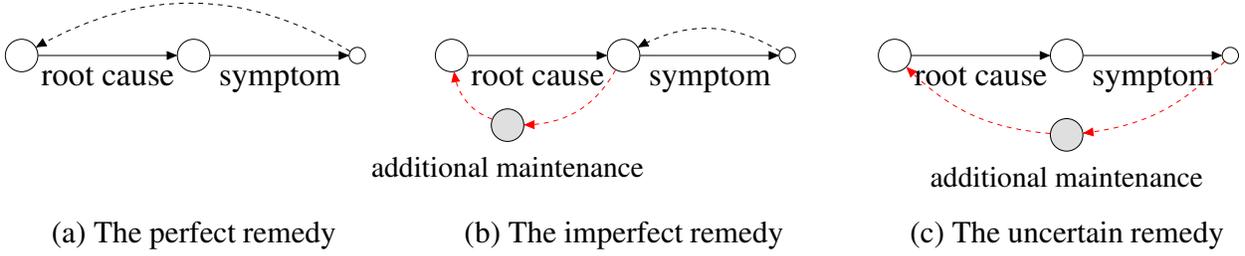

    \centering
    \centering
        \resizebox{\linewidth}{!}{
    \tikz{
    \node[latent,font=\fontsize{35}{0}\selectfont](w0){};
    \node[latent,right = 3cm of w0](w1){};
    \node[latent,right = 3cm of w1,scale=0.5](w2){};
    \node[below = 3cm of w1,scale=1.9](ind1){(a) The perfect remedy};
    \node[latent,right = 1.5cm of w2](w3){};
    \node[latent,right = 3cm of w3](w4){};
    \node[latent,right = 3cm of w4,scale=0.5](w5){};
    \node[obs,below left = 1cm and 2cm of w4](w6){};
    \node[below = 0.2cm of w6,font=\fontsize{18}{0}\selectfont](text){\text{additional maintenance}};
    \node[below = 3cm of w4,scale=1.9](ind2){(b) The imperfect remedy};
    \node[latent,right = 5.5cm of w5](w7){};
    \node[latent,left = 3cm of w7](w8){};
    \node[latent,right = 3cm of w7,scale=0.5](w9){};
    \node[obs,below = 1cm of w7](w10){};
    \node[below= 0.2cm of w10,font=\fontsize{18}{0}\selectfont](text2){\text{additional maintenance}};
    \node[below = 3cm of w7,scale=1.9](ind3){(c) The uncertain remedy};
    \path[->,draw]
    (w0)edge node[below,scale = 2]{root cause}(w1)
    (w1)edge node[below,scale = 2]{symptom}(w2)
    (w2)edge[bend right=30,dashed] node[above,scale=2,color=blue]{}(w0)

    (w3)edge node[below,scale = 2]{root cause}(w4)
    (w4)edge node[below,scale = 2]{symptom}(w5)
    (w5)edge[bend right=30,dashed] node[above,scale=2,color=blue]{}(w4)
    (w4)edge[bend left = 30,dashed,color=red] node[above,scale=3.5]{}(w6)
    (w6)edge[bend left = 30,dashed,color=red] node[above,scale=3.5]{}(w3)
    (w8)edge node[below,scale = 2]{root cause}(w7)
    (w7)edge node[below,scale = 2]{symptom}(w9)
    (w9)edge[bend left = 20,dashed,color=red] node[above,scale=3.5]{}(w10)
    (w10)edge[bend left = 20,dashed,color=red] node[above,scale=3.5]{}(w8)
    }
    }
    \caption{Three types of remedy: the perfect remedy, the imperfect remedy and the uncertain remedy. The solid paths model the deteriorating process, the dashed paths, i.e. the recovery paths, represent the maintenance process. The grey nodes represent the uncertain maintenance. }
    \label{fig:remedy}
\end{figure}

If the root cause is not remedied but only a subset of the secondary or intermediate faults are remedied, then after the intervention the status of the repaired component will not return to AGAN. We call such an intervention an \textit{imperfect remedy}. Figure \ref{fig:remedy}(b) depicts this maintenance process. The black dashed line pointing from the sink node to the interior node along the root-to-sink path. In order to fully restore the machine, additional maintenance is needed. This is the source of the uncertainty associated with the imperfect remedy. If an intervention is made at time $t$, then the necessary additional maintenance could not be known at that time. So in particular such maintenance would not be recorded in the maintenance log. We represent this necessarily subsequent maintenance by a grey vertex. As shown in Figure \ref{fig:remedy}(b), after this unobserved maintenance, the status of the component returns to AGAN as depicted by the red dashed line.  

If we cannot extract precisely what remedial maintenance was taken in an intervention recorded in the maintenance logs, then such intervention is classified as an \textit{uncertain remedy}. Diagnostic information has not yet been made available so the root cause of the failure cannot be determined. The recovery path is not observed, so there is no black dashed line pointing from the sink node to any node that lies upstream of it, see Figure \ref{fig:remedy}(c). A follow-up check and maintenance will be carried out in order to restore the broken part, which is represented by a grey interior vertex along the red path.


\subsection{Formulation of causal algebras for remedial intervention}
Based on the three types of remedy, we formally generalise the intervened regime into two categories: the \textit{perfect remedial intervention} regime and the \textit{random remedial intervention} regime. The former refers to a perfect remedy, whilst the latter indicates a regime arising when uncertainty is introduced to the intervened system because of the uncertain impact of the remedial action. Note that this regime includes the imperfect remedy and the uncertain remedy. These regimes are defined formally below.

When constructing the core event variables for the GN, we have one variable representing the maintenance, see Figure \ref{fig:global net}. Let $R$ denote this variable, $R\in \bm{U}$, whose state space $\mathbb{R}$ consists of different observed maintenance core events. We now embed some causal terminology used by reliability engineers within the semantics we have developed above. Let $A$ be the action variable representing the uncertain subsequent maintenance. The action variable takes values in $\mathbb{A}=\{a_1,...,a_{n_A}\}$, which is a known finite set. The effect of the observed maintenance is informed by a binary variable $\delta$, called a \textit{status indicator} and defined as 
\begin{equation}
    \delta=\begin{cases}
    1, & \text{ if the status is AGAN after maintenance,}\\
    0, & \text{ otherwise.}
    \end{cases}
\end{equation} The value of $\delta$ determines whether the subsequent maintenance $a\in \mathbb{A}$ is required. 
Some root causes are corrected by a remedial intervention. Let $\bm{w}_{R}=\{w_{l_1},...,w_{l_n}\}$ denote the set of $n$ positions corresponding to all root causes. We define a new variable over $\bm{w}_{R}$ to indicate which positions are affected. This is called an \textit{intervention indicator} and denoted by $\bm{I}=(I_{w_{l_1}},...,I_{w_{l_n}})$. For $i\in\{1,...,n\}$,
\begin{equation}
    I_{w_{l_i}} = \begin{cases}
    1, & \text{ if the root cause represented on $w_{l_i}$ is fixed by the maintenance, }\\
    0, & \text{ otherwise.}
    \end{cases}
\end{equation}

Based on the difference between perfect and random remedial intervention regimes, the intervention indicator $\bm{I}$ can be constructed from a two-component mixture model. Let $q\in[0,1]$ be the probability of $\delta=1$. Thus, for $r\in\mathbb{R}$,
\begin{equation}
    p(\bm{I}|r)= q\times p(\bm{I}|r, \delta=1) + (1-q)\times p(\bm{I}|r, \delta=0).
\end{equation}

For a random remedial intervention regime, $(r,a)$ should provide sufficient information to identify its associated root cause. Translating this assertion into a conditional independence statement gives the next assumption.
\begin{assumption}
The intervention indicator is independent of $\lambda$ and the status indicator conditional on the observed and unobserved maintenance.
\begin{equation}
  \bm{I}\indep (\lambda,\delta)|R, A.
\end{equation}
\end{assumption}
One more assumption is made concerning the action variable.
  \begin{assumption}
   The uncertain subsequent maintenance depends on the observed deteriorating process and the observed maintenance.
  \end{assumption}
  
 Under these assumptions we can now express the distribution for the intervention indicator under the random remedial intervention. Thus, 
\begin{equation}\label{equ:I}
       \begin{split}
         p(\bm{I}|r, \delta=0)&\propto p(\bm{I}|r,a, \delta=0)p(a|r, \delta=0)  \\
         &\propto p(\bm{I}|r,a)p(a|\lambda,r, \delta=0)p(\lambda|r, \delta=0).
       \end{split} 
\end{equation} 

The objective behind correcting a root cause is to prevent the fault caused by it reoccurring. This implies that after a remedial intervention, the probability distribution over root causes needs to be transformed from what it was in the idle system. This can be represented as a stochastic manipulation on the idle system. 
More formally the effect of this stochastic manipulation under intervention $r$ can be represented by the map:
$$Z_{r}:(C,\bm{\theta})\mapsto(\hat{C},\bm{\hat\theta}).$$ The domain of this map is the topology and primitive probabilities of the CEG: $V\bigotimes E\bigotimes\Theta$. Under this transformation, the topology of the CEG is assumed to be unchanged after intervention: $\hat C=C$.\footnote{Note that although the topology of the CEG is unchanged, the stages or positions may change as a result of manipulation.}

If a root cause $x_i$ has been corrected by an intervention, then the transition probability along the edge $e(x_i)$ needs to be reduced because $x_i$ has been corrected and is unlikely to reoccur immediately after the maintenance. Within our Bayesian model, let $f(\cdot)$ denote the prior distribution over the primitive probabilities $\bm{\theta}_w$. Suppose the hyperparameters are $\bm{\alpha}_w$, then 
\begin{equation}
    \bm{\theta}_w\sim f(\bm{\theta};\bm{\alpha}_w).
\end{equation} We can now represent an intervention on the system through defining a map that will transform $\bm{\alpha}_w$. Given $r$, let the intervention indicator be denoted by $\bm{I}(r)$. Let $w(r)$ represent the set of vertices so that any $w\in w(r)$ satisfies $I_w(r)=1$. The parent nodes of these vertices are represented by $pa(w(r))$. One feature of remedial intervention is that multiple root causes could be subjected to simultaneous interventions. 

Denote the post-intervention hyperparameter vector by $\hat{\bm{\alpha}}$. We define a generic function $\kappa$ to update the hyperparameters:
\begin{equation}
    \kappa:(\bm{\alpha}_{pa(w(r))},\bm{I}(r))\mapsto \hat{\bm{\alpha}}_{pa(w(r))},
\end{equation}where $\kappa$ will be determined by domain experts informed by maintenance history \citep{Yu2020,Yu20211}. One example of this map is a linear transformation:
\begin{equation}
    \hat{\bm{\alpha}}_{pa(w(r))} = \bm{\alpha}_{pa(w(r))} + \omega(1-\bm{I}(r)),
\end{equation}where $\omega>0$ is introduced to the map to control the effect of $r$ on the hyperparameters so that the likelihood of each root cause can be adjusted \citep{Yu2020}. We note a similar map was introduced into the semantics of BNs by Eaton and Murphy \citep{Eaton2007} calling it a relaxation of Pearl's \textit{do}-algebra and defining it on a BN where they added intervention indicators \citep{Dawid2002} to signal changes of hyperparameters.



\section{The causal identifiability of a remedial intervention on CEG}\label{sec:identifiability}

The probability distribution over every floret $F\in F(pa(w(r)))$ is manipulated as a result of the remedial intervention. This means every root-to-leaf path in $\Lambda=\Lambda_{pa(w(r))}$ is affected by the intervention, while all root-to-leaf paths in $\overline{\Lambda}=\Lambda(C)/\Lambda$ are not affected.

 Let $C^{\Lambda}$ be a sub-CEG where for every root-to-sink path $\lambda\in \Lambda(C^{\Lambda})$, $\lambda\in\Lambda$; and for every $\lambda\in\Lambda$, we have $\lambda\in \Lambda(C^{\Lambda})$. Similarly, let $ C^{\overline{\Lambda}}$ represent a sub-CEG where for every $\lambda\in \Lambda( C^{\overline{\Lambda}})$, $\lambda\in\overline{\Lambda}$; and for every $\lambda\in\overline{\Lambda}$, we have $\lambda\in\Lambda(C^{ \overline{\Lambda}})$.
Let $\pi^{\Lambda}(\cdot)$ and $\pi^{\overline{\Lambda}}(\cdot)$ denote the probability mass functions over $C^{\Lambda}$ and $C^{\overline{\Lambda}}$ respectively.



Under remedial intervention, the controlled d-events are root causes that are potentially affected by the remedy. Let the controlled d-events be ``$x$''. Then the probability of passing $w\in w(x)$ is manipulated given $r$, where $w(x)=ch(pa(w(r)))$. Note that due to the asymmetry of the CEG, $w(x)$ may not be a fine cut \citep{Wilkerson2020}. Let $\overline{w}=\bm{w}_R/w(x)$.

Analogously to Thwaites's \citep{Thwaites13} back-door theorem for a singular manipulation on a CEG, we first define some notations for the analogous back-door criterion for remedial intervention on the CEG. The effect d-events, denoted by ``$y$'', can be a failure or a specific symptom. We represent all the root-to-sink paths corresponding to such an effect by $\Lambda_{w(y)}$. The set $\Lambda_{w(z)}$ for some d-events ``$z$'' is analogous to the back-door partition set \citep{Pearl2000} and is chosen such that the back-door criterion shown in Theorem \ref{backdoor ceg} is satisfied. 

Let $\pi^*(pa(w(r))|r)$ denote the post-intervention distribution over $pa(w(r))$, and for $w\in w(x)$,
\begin{equation}
\pi_{\Lambda_{w}}^*=\pi(\Lambda_{w}||\pi^*(pa(w(r))|r)),
\end{equation}
\begin{equation}
    \pi^{\Lambda}_{y}=\pi^{\Lambda}(\Lambda_{w(y)}||\pi^*(pa(w(r))|r)), 
\end{equation}
\begin{equation}
    \pi^{\overline{\Lambda}}_{y}=\pi^{\overline{\Lambda}}(\Lambda_{w(y)}||\pi^*(pa(w(r))|r)). 
\end{equation}

The theorem below now formulates the back-door criterion for the stochastic manipulation from remedial intervention on the CEG.
\begin{theorem}\label{backdoor ceg}
Given the remedial intervention $r$, the causal effect of the manipulation $do(\pi^*(pa(w(r))|r))$ is given by
\begin{equation}\label{eq:backdoor}
    \pi(\Lambda_{w(y)}||\pi^*(pa(w(r))|r))
    = \pi^{\Lambda}_{y} + \pi^{\overline{\Lambda}}_{y},
\end{equation} where the probabilities
\begin{equation}
  \pi^{\Lambda}_{y} =\sum_{w\in w(x)}\sum_z \pi(\Lambda_{w(y)}|\Lambda_{w},\Lambda_{w(z)})\pi(\Lambda_{w(z)})\pi^*_{\Lambda_{w}},
\end{equation}
\begin{equation}
    \pi^{\overline{\Lambda}}_{y}=\sum_{w'\in \overline{w}}\pi(\Lambda_{w(y)}|\Lambda_{w'})\pi(\Lambda_{w'}).
\end{equation}
\end{theorem}

\begin{proof}
The two components in equation \ref{eq:backdoor} correspond to $C^{\Lambda}$ and $C^{\overline{\Lambda}}$ respectively. As mentioned above, the probability of passing along $\Lambda_{w(y)}$ within $C^{\overline{\Lambda}}$ is the same as that in the idle system. We will therefore only focus on the first term on the RHS of the equation. 

 The probability of a root-to-sink path passing the vertex $w\in w(x)$ is reassigned when the post-intervention distribution $\pi^*(pa(w(r))|r)$ is imported to the system. Then we have
\begin{equation}\label{eq:theorem4}
\begin{split}
      \pi^{\Lambda}_{y}& =\sum_{w\in w(x)}\pi(\Lambda_{w(y)}||\Lambda_{w})\pi(\Lambda_{w}||\pi^*(pa(w(r))|r))\\
   & =\sum_{w\in w(x)}\pi(\Lambda_{w(y)}||\Lambda_{w})\pi_{\Lambda_{w}}^*.
   \end{split}
   \end{equation}
 As proved by Thwaites \citep{Thwaites13} for a singular manipulation, we have:
 \begin{equation}
     \pi(\Lambda_{w(y)}||\Lambda_{w}) = \sum_z \pi(\Lambda_{w(y)}|\Lambda_{w},\Lambda_{w(z)})\pi(\Lambda_{w(z)}).
 \end{equation}
Applying this result to equation \ref{eq:theorem4}, we then have that
\begin{equation}
\pi^{\Lambda}_{y}
= \sum_{w\in w(x)}\sum_z \pi(\Lambda_{w(y)}|\Lambda_{w},\Lambda_{w(z)})\pi(\Lambda_{w(z)})\pi^*_{\Lambda_{w}}.
\end{equation}
\end{proof}

Therefore if the post-intervention distribution $\pi^*(pa(w(r)|r)$ is known given $r$, the effect of a perfect remedial intervention can be represented as 
\begin{equation}
    \pi(\Lambda_{w(y)}|do(r))=\pi(\Lambda_{w(y)}||\pi^*(pa(w(r))|r)),
\end{equation} which has been shown to be identifiable in Theorem \ref{backdoor ceg}.

Under a random remedial intervention, the value of $\bm{I}(r)$ is no longer deterministic. The intervention indicator should be drawn from a distribution conditional on the additional maintenance $a$ as shown in equation \ref{equ:I}. So the effect of a random remedial intervention can be identified through the formula:
\begin{equation}
    \begin{split}
        \pi(\Lambda_{w(y)}|do(r)) & = \sum_{a}\pi(\Lambda_{w(y)},a|do(r))\\
        & = \sum_a\pi(\Lambda_{w(y)}|do(r),a)p(a|do(r))\\
        & = \sum_a\pi(\Lambda_{w(y)}||\pi^*(pa(w(r,a))|r))p(a|r),
    \end{split}
\end{equation} where $\pi^*(pa(w(r,a))|r)$ is assumed to be known given $r$ and $a$.

\section{Missingness}\label{sec:missing}Before establishing the framework of missingness for the GN-CEG causal model, we make the following assumption:


\begin{assumption}\label{complete}
 The CEG $C$ is faithfully constructed so that there is no domain knowledge that has not been expressed within the tree. 
\end{assumption}


On the basis of these two assumptions, we summarise the following three categories of missingness for our model. Type 1 missingness is the missingness of the time spent on each core event and the time taken by a transition from one position to another on the underlying CEG. The path on the CEG is not observed either. Both Type 2 and Type 3 missingness involve missing core events. Type 2 missingness is the missingness of core events that causes difficulties in learning the full latent root-to-sink path on the CEG. Type 3 missingness, on the other hand, refers to the case when some core events are missing but a root-to-sink path can still be learned. 

Note that if we still use the aforementioned function $Q$ to learn a path on the CEG, when Type 2 missingness occurs, there could be multiple root-to-sink paths that are associated to a subgraph $G$ on $C$. In this case, $Q$ might not be well-defined. To solve this problem, we next introduce a CEG with missing indicators, called an \textit{M-CEG}. The M-CEG enables us to identify a single root-to-sink path for each document, even when the core events are partially observed. 
\subsection{Latent time and states on the CEG}
We only observe the total failure time or total deteriorating time of a system from the dataset. The time interval between any two consecutive core events is unobservable. In our hierarchical model, we simply ignore this time and do not estimate it in our analysis. However, this does not mean that the core events are discrete-time. Let $T_i$ denote the holding time for position $w_i\in V_C$. Barclay \textit{et.al} \citep{Barclay2015} introduced holding time distributions in the CEG so that the holding time depends only on the current and the receiving positions. We can prove the following conditional independence statement for the holding time. 
\begin{proposition} Given a consistent assignment $H$, the holding time $T_i$ for every $w_i\in V_C$ satisfies:
\begin{equation}
    T_i\indep \bm{U}^H_i|Y^H(w_i).
\end{equation} 
\end{proposition}

\begin{proof}
Here we aim to show for every $U_{i,k}\in\bm{U}_i^H$, $Y^H(w_i)$ d-separates $T_i$ and $U_{i,k}$, \textit{i.e.} $T_i\indep U_{i,k}|Y^H(w_i)$. The parent node of $U_{i,k}$ within the GN can be in the same community as $U_{i,k}$ or in another community, say $\bm{U}_j^H$ such that $w_j \in pa^H(w_i)$. Below we discuss these two scenarios separately.

If $pa^H(U_{i,k})\subseteq \bm{U}^H_i$, from the flattening graph of the nested network, there are two types of paths between $T_i$ and $U_{i,k}$:
\begin{enumerate}
    \item $T_i\leftarrow Y^H(w_i)\rightarrow U_{i,k}$,
    \item $T_i\leftarrow Y^H(w_i)\rightarrow pa^H(U_{i,k})\rightarrow U_{i,k}$.
\end{enumerate} There are no collars along these paths. Both paths can be blocked by $Y^H(w_i)$, therefore, $T_i\indep U_{i,k}|Y^H(w_i)$.

If $pa^H(U_{i,k})\subset \bm{U}_j^H$, where $w_j$ is the parent of $w_i$ so that $Y^H(w_j)$ instantiated given $H$, then there exists another type of paths:
\begin{enumerate}
\setcounter{enumi}{2}
    \item $T_i\leftarrow Y^H(w_i)\leftarrow I^H(w_i)\leftarrow Y^H(w_j)\rightarrow pa^H(U_{i,k})\rightarrow U_{i,k}$.
\end{enumerate} This path is also blocked by $Y^{H}(w_i)$. Therefore we always have $T_i\indep \bm{U}_i^H|Y^H(w_i)$.
\end{proof}

As an immediate result, we have the following corollary.

\begin{corollary}
Given a consistent assignment $H$, the holding time $T_i$ is independent of the number of observations that trigger a transition from the vertex $w_i\in V_C$ conditional on $Y^H(w_i)$.
\end{corollary}

\subsection{Floret-dependent missingness and M-CEG}
In Barclay \textit{et.al} \citep{Barclay2014}, missing at random (MAR), missing completely at random (MCAR) and missing not at random (MNAR) \citep{rubin1976} on the CEG were discussed. This motivated a second type of missingness in a fault analysis we call the \textit{floret-dependent missingness}. 

There are two main factors contributing to the floret-dependent missingness. The first one is the unobservability of $F(w)$. Note that some core events will never be missing, such as the component's name. In contrast, the root causes are highly likely to be unobserved. The florets representing root causes, $F(pa(\bm{w}_R))$, are therefore highly likely to be missing. We incorporate the unobservability of $F(w)$ by introducing missing indicators to the CEG.

We therefore define a new CEG with missing indicators, call this a \textit{missingness CEG} (M-CEG), and denote this by $C^M$. An event tree with missing indicators is called an \textit{M-event tree}, denoted by $T^M$. The M-CEG is derived from the corresponding $T^M$ given $C$. Call this a \textit{fact CEG} (F-CEG) with the \textit{fact event tree} $T$.  


\begin{definition} If the information represented on the floret $F(v)$ is unobservable for $v\in V_T$ , then we define a \textit{missing floret indicator} $B_v$ so that: 
\begin{equation}
B_v=
\begin{cases} 1, \text{ if $F(v)$ is missing,}\\
0, \text{ otherwise.} 
\end{cases} 
\end{equation}
\end{definition}
To import this binary variable into the event tree, a new floret $F(v')=(v',E(v'))$, whose edge set is $E(v')=\{e_{v',v''},e_{v',v}\}$, is constructed to represent $B_v$. The edge $e_{v',v''}$ entering $v''\in V_T$ is labelled as ``missing''. The edge $e_{v',v}$ entering $F(v)$ is labelled as ``non-missing''. An example is shown in Figure \ref{fig:example m-ceg}(a).

\begin{definition} An \textit{M-CEG topology} $C^M=(V_{C^M},E_{C^M})$ is derived from an M-event tree $T^M$ as a function of the fact CEG $C$ and the fact tree $T$. The vertices are classified into four types. For any vertex $w\in V_{C^M}$, if the information represented by $F(w)$ is unobservable, then $w\in V^{M}$. If the information represented by $F(w)$ is always observable, then $w\in V^{O}$. If $w$ corresponds to the missing floret indicator, then $w\in V^{MI}$. If it is none of these, then $w$ must be a terminal node $w\in V^S$. The vertex set is therefore $$V_{C^M}=V^M\cup V^O\cup V^{MI}\cup V^S.$$
\end{definition}

The second factor contributing to the floret-dependent missingness is human error. The field engineers have all been trained before doing regular checks and filling forms. But we cannot ensure that every field engineer has a perfect grasp of all the domain knowledge. For example, if the field engineer has limited understanding about the root cause, then he will not be capable of recording it in detail. This is therefore a knowledge gap between the domain experts and the field engineers. 

Therefore it is important that this heterogeneity in engineers should be part of our model. We partition all the engineers into $N_e$ disjoint clusters $\{l_1,...,l_{N_e}\}$, where $0<N_e<\infty$. For each cluster $l_i$, $i\in \{1,...,N_e\}$, we assign a heterogeneity parameter to reflect the competence of engineers in this group. Let $\bm{\eta}=\{\eta_1,...,\eta_{N_e}\}$ denote the set of heterogeneity parameters. We encode the heterogeneity of engineers to the M-CEG by assigning a prior probability mass function $p(\eta_j)$ to every heterogeneity parameter $\eta_j$ and drawing the transition probabilities from a mixture model:
\begin{equation}\label{equ:16}
    \bm{\theta}_{ w}\sim\sum_{j=1}^{N_e}f(\bm{\theta}_w;\bm{\alpha}_w,\eta_j)p(\eta_j).
\end{equation}

\subsection{Event-dependent missingness}

For any $w_i\in V^{M}$, suppose the floret variable $Y(w_i)=y_{w_i,w_j}\neq 0$ so that there is a transition on the M-CEG along $e_{w_i,w_j}$. Given an assignment $H$, let $\bm{U}^H_{i,j}=\{U_{1,ij}^H,...,U_{n,ij}^H\}$, $n>0$, denote the community of core event variables that trigger this transition. 

It is likely that $\bm{U}^H_{i,j}$ is partially observed by the field engineers. We assume that as long as one core event variable $U_r\in \bm{U}^H_{i,j}$ is observed, the transition along $e_{w_i,w_j}$ can still be triggered. 

We firstly define a missing indicator for every core event variable.

\begin{definition} For every core event variable $U_{k,ij}\in\bm{U}^H_{i,j}$, we define a \textit{missing event indicator} $B_{k,ij}^H$ so that
\begin{equation}
    B_{k,ij}^H=
\begin{cases} 1, \text{\hspace{6mm}if $U_{k,ij}^H$ is missing,}\\
0, \text{\hspace{6mm}otherwise.} 
\end{cases} 
\end{equation}
 The set of missing event indicators associated to $\bm{U}^H_{i,j}$ that trigger a transition from $w_i$ to $w_j$ is then given by
\begin{equation}
   \bm{B}^H_{i,j}=\{B_{1,ij}^H,\cdots,B_{n,ij}^H\}.
\end{equation}
\end{definition}

The missing event indicators can be added to the GN-CEG model provided we make the following conditional independence assumptions. Note that the missing event indicators lie within the GN.
\begin{assumption}
Given a consistent assignment $H$, the missing event indicator $B_{k,ij}^H$ shares the same direct superior as the core event variable $U_{k,ij}^H$.
\end{assumption}
Under this assumption, we have \begin{equation}
 Dsup^H(B_{k,ij}^H)=Y^H(w_i).   
\end{equation}
\begin{assumption}\label{assumption N}
Given a consistent assignment $H$, the parent variables of a missing event indicator $B_{k,ij}^H$ on the GN consist of the $N$ core event variables preceding $U_{k,ij}^H$. We call this missingness the \textit{N-event-dependent missingness}. Let $N >0$ where we allow for the possibility that $N\geq k$. 
\end{assumption} 

When $N\geq k$, the core event variables whose direct superior is not $Y^H(w_i)$ can also be informative about the observability of $U_{k,ij}^H$. Under assumption \ref{assumption N}, the set of parent variables of $B_{k,ij}^H$ that lie on the GN is
\begin{equation}
pa^H(B_{k,ij}^H)=\{U_{k-1,ij}^H,...,U_{k-N,ij}^H\}.
\end{equation}Applying the RCMC implies
\begin{equation}
   B_{k,ij}^H\indep nd^H(B_{k,ij}^H)|Y^H(w_i), pa^H(B_{k,ij}^H).
\end{equation}The set of parent variables of $B_{k,ij}^H$ in the flattening $H^{\downarrow}$ is
\begin{equation}
   pa^{H^{\downarrow}}(B_{k,ij}^H) =\{Y^H(w_i), U_{k-1,ij}^H,...,U_{k-N,ij}^H\}.
\end{equation}

As with the floret-dependent missingness, the heterogeneity in engineers is also responsible for the event-dependent missingness. We can now assign a distribution to the missing event indicator conditional on the parent variables and the heterogeneity parameter: $p(b_{k,il}^H|pa^{H^{\downarrow}}(B_{k,il}^H),\eta_j)$. 

Figure \ref{fig:MRCG} shows an example of illustrating graphically the nested causal network with missing event indicators $B_1$, $B_2$ and $B_3$. This is a 1-event-dependent missingness, where $B_i$ depends on $U_{i-1}$ and $Y_{w_0}=y_{w_0,w_1}$.
\begin{figure}[H]
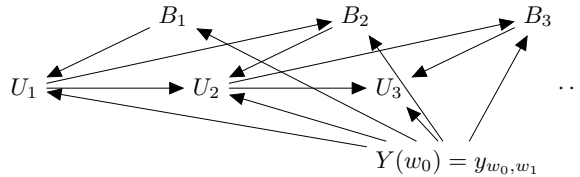

    \centering
    \resizebox{0.6\linewidth}{!}{
    \tikz{
\node[font=\fontsize{10}{0}\selectfont](u1){$U_1$};
\node[right = 2cm of u1,font=\fontsize{10}{0}\selectfont](u3){$U_2$};
\node[right = 2cm of u3,font=\fontsize{10}{0}\selectfont](u4){$U_3$};

\edge{u3}{u4};
\node[below right = 0.5cm and 2cm of u3,font=\fontsize{10}{0}\selectfont](y1){$Y(w_0)=y_{w_0,w_1}$};
\edge{y1}{u1,u3,u4};
\node[above right = 0.5cm and 1.5cm of u1,font=\fontsize{10}{0}\selectfont](b1){$B_1$};
\node[above right = 0.5cm and 1.5cm of u3,font=\fontsize{10}{0}\selectfont](b3){$B_2$};
\node[above right = 0.5cm and 1.5cm of u4,font=\fontsize{10}{0}\selectfont](b4){$B_3$};
\edge{u1}{u3};
\edge{y1}{b1,b3,b4};
\edge{b1}{u1};
\edge{u1}{b3};
\edge{b3}{u3};
\edge{b4}{u4};
\edge{u3}{b4};
\node[right = 2cm of y1,font=\fontsize{10}{0}\selectfont](quot1){$\cdots$};
\node[right = 2cm of u4,font=\fontsize{10}{0}\selectfont](quot2){$\cdots$};
    }
    }
    \caption{Adding the missing event indicators to the flatenning of the nested causal network.}
    \label{fig:MRCG}
\end{figure}

\subsection{Identifiability of a causal effect with missingness}

We now study causal identifiability on the M-CEG when there is floret-dependent missingness. This is analogous to the causal identifiablity on the BN with missingness indicators \citep{Saadati2019}, which is called an m-graph \citep{mohan2013,mohan2014,mohan2018}. 

Let $\bm{R}^O$ denote the variables that are observed in all data cases, $\bm{R}^M$ denote the partially observed variables and $\bm{B}$ denote the missingness indicators. Let $\bm{R}=\bm{R}^O\cup\bm{R}^M$.

Mohan and Pearl \citep{mohan2013,mohan2014} defined recoverability of a joint probability $p(X)$ under different missingness mechanisms on an m-graph. Here a joint probability distribution is recoverable whenever it can be consistently estimated despite the missingness mechanism. As shown by Mohan \textit{et.al} \citep{mohan2013}, the joint distribution is always recoverable when data are MCAR or MAR. The recoverability of the joint distribution is complicated to analyse when data are MNAR because the missing indicators are not independent of $\bm{R}^O$ and $\bm{R}^M$. However, even when the joint distribution cannot be consistently estimated, the probability $p(y|do(x))$ is still estimable from the dataset with missing data \citep{mohan2014}. The recoverability of this probability is a sufficient condition for the identifiability of the causal query in the subgraph comprising of vertices corresponding to $\bm{R}^O$ and $\bm{R}^M$ \citep{mohan2018}.

Saadati and Tian \citep{Saadati2019} have developed a back-door criterion and an adjustment criterion for identifying and recovering causal effects from missing data on an m-graph. The definition is given below. 

\begin{definition}[M-Adjustment Formula \citep{Saadati2019}] Given an m-graph G over $\bm{R}$ and $\bm{B}$, a set $\bm{Z}\subset \bm{R}$ is called an m-adjustment (adjustment under missing data) set for estimating the causal effect of $X$ on $Y$, if, for every model compatible with $G$, and $\bm{W}=\bm{R}^M\cap(X\cup Y\cup \bm{Z})$, then
\begin{equation}
    P(y|do(x))=\sum_{\bm{z}} P(y|x,\bm{z},\bm{B}_{\bm{W}}=0)P(\bm{z}|\bm{B}_{\bm{W}}=0).
\end{equation}
\end{definition}

In this formula, the sufficient condition for $\bm{Z}$ to be an m-adjustment set is that it satisfies the adjustment criterion, $Y\indep \bm{B}_{\bm{W}}|X,\bm{Z}$ and $\bm{Z}\indep \bm{B}_{\bm{W}}$. Note that the adjustment criterion generalises the back-door criterion to the complete graph. As pointed out by \citep{shpitser2012}, if $\bm{Z}$ satisfies the back-door criterion with
respect to $(X, Y)$ in G, then $\bm{Z}$ satisfies the adjustment
criterion with respect to $(X, Y)$ in G. From Thwaites and Smith's work \citep{Thwaites10,Thwaites13}, an analogous back-door partition $\Lambda_{w(z)}$ can be identified from the CEG. So here we focus on adapting the back-door theorem for the M-CEG. 

We begin with the back-door theorem for a singular manipulation on $C^M$ and then extend this to a remedial intervention. Here we denote the probability mass function defined on $C^M$ by $\pi^M(\cdot)$.

As for the F-CEG, when observing a d-event $x$, the transitions along $e(x)$ to $w(x)$ are triggered on the M-CEG. The edges $e(x)$ lie in the florets $F(pa(w(x)))$, so the missing floret indicators are $B_{pa(w(x))}$. Let $w(B_{pa(w(x))}=1)$ denote the vertices in $C^M$ corresponding to $B_{pa(w(x))}=1$. Let $\Lambda_{w(B_{pa(w(x))}=0)}$ denote the set of root-to-sink paths passing through $w(B_{pa(w(x))}=0)$. Here we still use $x$, $y$, $z$ to represent respectively the control, effect d-events and the d-events whose corresponding nodes on the M-CEG constitute the back-door block set. Let $w(z)\subset V^O\cup V^M$. Let $$\bm{W}=V^M\cap(pa(w(x))\cup pa(w(y))\cup pa(w(z))).$$

We can now derive the M-back-door criterion for a singular manipulation on an M-CEG.
\begin{theorem}\label{theorem:mceg}
 The set of paths $\Lambda_{w(z)}$ is an M-back-door partition on $C^M$ for estimating the causal effect of $\Lambda_{w(x)}$ on $\Lambda_{w(y)}$ whenever \begin{equation}
     \pi(\Lambda_{w(y)}||\Lambda_{w(x)}) = \sum_z \pi^M(\Lambda_{w(y)}|\Lambda_{w(x)},\Lambda_{w(z)},\Lambda_{w(\bm{B}_{\bm{W}}=0)})\pi^M(\Lambda_{w(z)}|\Lambda_{w(\bm{B}_{\bm{W}}=0)}).
\end{equation}
\end{theorem}

Combining Theorem \ref{backdoor ceg} and Theorem \ref{theorem:mceg}, we can now state the following M-back-door theorem for a remedial intervention on $C^{M}$. Let
\begin{equation}
\pi_{\Lambda_{w}}^{M*}=\pi^M(\Lambda_{w}|\Lambda_{w(\bm{B}_{\bm{W}}=0)}, do(\pi^*(pa(w(r))|r))).
\end{equation}
\begin{theorem} The causal effect of a remedial intervention $r$ is identifiable when the M-back-door partition $\Lambda_{w(z)}$ satisfies equation \ref{eq:backdoor} in Theorem \ref{backdoor ceg} and the two terms on the RHS of this equation can be expressed as:
\begin{equation}
    \pi^{\Lambda}_{y}= \sum_{w\in w(x)}\sum_z \pi^M(\Lambda_{w(y)}|\Lambda_{w},\Lambda_{w(z)},\Lambda_{w(\bm{B}_{\bm{W}}=0)})\pi^M(\Lambda_{w(z)}|\Lambda_{w(\bm{B}_{\bm{W}}=0)})\pi^{M*}_{\Lambda_{w}},
\end{equation}
\begin{equation}
    \pi^{\overline{\Lambda}}_{y}= \sum_{w'\in \overline{w}}\pi^M(\Lambda_{w(y)}|\Lambda_{w'},\Lambda_{w(B_{pa(w(y))}=0)},\Lambda_{w(B_{pa(w')}=0)})\pi^M(\Lambda_{w'}|\Lambda_{w(B_{pa(w')}=0)}).
\end{equation}
\end{theorem}

\begin{figure}
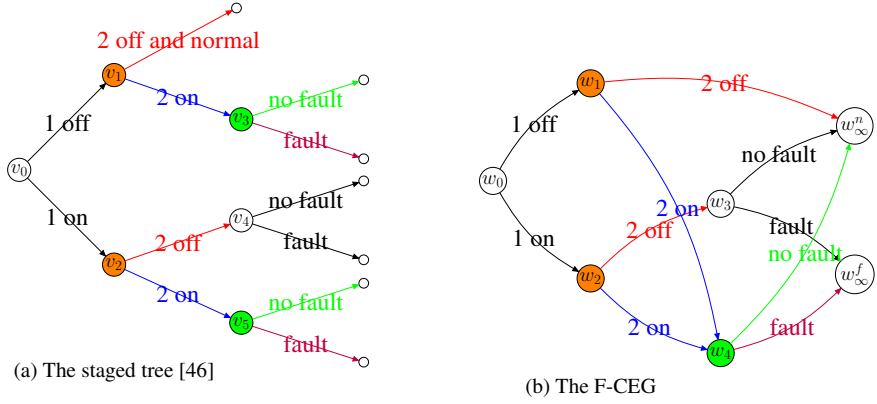

    \resizebox{0.7\linewidth}{!}{
    \tikz{
    \node[latent,font=\fontsize{20}{0}\selectfont](v0){$v_0$};
    \node[latent,above right = 4cm of v0,fill=orange,font=\fontsize{20}{0}\selectfont](v1){$v_1$};
    \node[latent,below right = 4cm of v0,fill=orange,font=\fontsize{20}{0}\selectfont](v2){$v_2$};
    \node[latent,scale = 0.5,above right = 2cm and 4cm of v1](s1){};
    \node[latent,below right = 1cm and 4cm of v1,fill=green,font=\fontsize{20}{0}\selectfont](v3){$v_3$};
    \node[latent,above right = 1cm and 4cm of v2,font=\fontsize{20}{0}\selectfont](v4){$v_4$};
    \node[latent,below right = 1.5cm and 4cm of v2,fill=green,font=\fontsize{20}{0}\selectfont](v5){$v_5$};
    \node[latent,above right = 1cm and 4cm of v3,scale=0.5](s2){};
    \node[latent,below right = 1cm and 4cm of v3,scale=0.5](s3){};
    \node[latent,above right = 1cm and 4cm of v4,scale=0.5](s4){};
    \node[latent,below right = 1cm and 4cm of v4,scale=0.5](s5){};
    \node[latent,above right = 1cm and 4cm of v5,scale=0.5](s6){};
    \node[latent,below right = 1cm and 4cm of v5,scale=0.5](s7){};
    \node[below = 3cm of v2,font=\fontsize{22}{0}\selectfont](t1){\text{(a) The staged tree \citep{Thwaites13}}};
    
    \node[latent,right = 4cm of s4,font=\fontsize{22}{0}\selectfont](w0){$w_0$};
    \node[latent, above right = 4cm of w0,font=\fontsize{22}{0}\selectfont,fill=orange](w1){$w_1$};
    \node[latent,below right = 4cm of w0,font=\fontsize{22}{0}\selectfont,fill=orange](w2){$w_2$};
    \node[latent, above right = 2cm and 4cm of w2,font=\fontsize{22}{0}\selectfont](w3){$w_3$};
    \node[latent,below right = 2cm and 4cm of w2,font=\fontsize{22}{0}\selectfont,fill=green](w4){$w_4$};
    \node[latent,above right = 2cm and 4cm of w4,font=\fontsize{22}{0}\selectfont](wf){$w_\infty^f$};
    \node[latent,above = 4cm of wf,font=\fontsize{22}{0}\selectfont](wn){$w_\infty^n$};
    \node[below = 3cm of w2,font=\fontsize{22}{0}\selectfont](t2){\text{(b) The F-CEG}};
    \path[->,draw]
    (v0)edge node[scale=2.5]{\text{1 off}} (v1)
    (v0)edge node[scale=2.5]{\text{1 on}} (v2)
    (v1)edge[color=red] node[scale=2.5]{2 off and normal} (s1)
    (v1)edge[color=blue] node[scale=2.5]{2 on} (v3)
    (v3)edge[color=green] node[scale=2.5]{no fault} (s2)
    (v3)edge[color=purple] node[scale=2.5]{fault} (s3)
    (v2)edge[color=red] node[scale=2.5]{2 off} (v4)
    (v2)edge[color=blue] node[scale=2.5]{2 on} (v5)
    (v4)edge node[scale=2.5]{no fault} (s4)
    (v4)edge node[scale=2.5]{fault} (s5)
    (v5)edge[color=green] node[scale=2.5]{no fault} (s6)
    (v5)edge[color=purple] node[scale=2.5]{fault} (s7)
   (w0)edge[bend left=15] node[scale=2.5]{\text{1 off}} (w1)
    (w0)edge[bend right = 15] node[scale=2.5]{\text{1 on}} (w2)
    (w2)edge[bend left=15,color=red] node[scale=2.5]{\text{2 off}}(w3)
    (w2)edge[bend right = 15,color=blue] node[scale=2.5]{\text{2 on}}(w4)
    (w1)edge[bend left = 15,color=blue] node[scale=2.5]{\text{2 on}}(w4)
    (w1)edge[bend left=15,color=red] node[scale=2.5]{\text{2 off}}(wn)
    (w3)edge[bend left=15] node[scale=2.5]{\text{no fault}}(wn)
    (w3)edge[bend left=15] node[scale=2.5]{\text{fault}}(wf)
    (w4)edge[bend right=15,color=green] node[scale=2.5]{\text{no fault}}(wn)
    (w4)edge[bend right=15,color=purple] node[scale=2.5]{\text{fault}}(wf)
    }}

    \caption{The staged tree and its corresponding CEG for the example in Section \ref{sec:missing}. The staged tree depicted in (a) has stages: $\{v_0\},\{v_1,v_2\},\{v_3,v_5\},\{v_4\}$ and positions: $\{v_0\},\{v_1\},\{v_2\},\{v_3,v_5\},\{v_4\}$. The stages in the F-CEG are $\{w_0\},\{w_1,w_2\},\{w_3,w_5\},\{w_4\}$.}
    \label{fig:example f-ceg}
\end{figure}
\begin{figure}
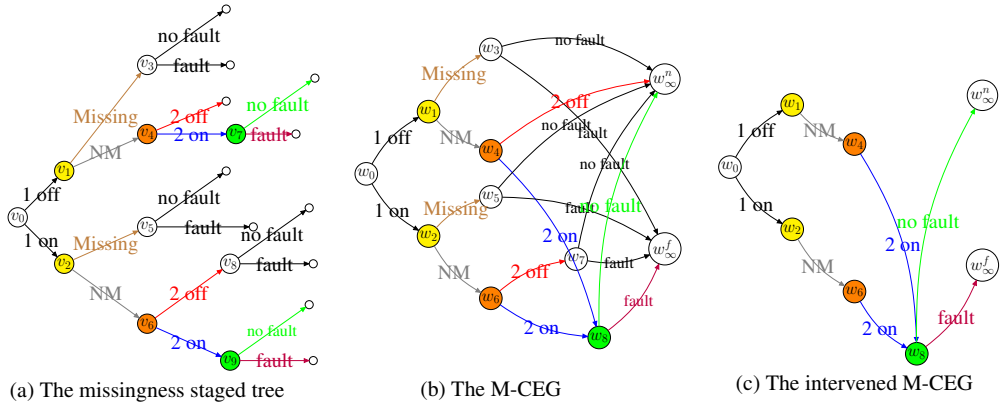

 \centering
    \resizebox{0.8\linewidth}{!}{
    \tikz{
    \node[latent,font=\fontsize{20}{0}\selectfont](v0){$v_0$};
    \node[latent,above right = 2cm of v0,font=\fontsize{20}{0}\selectfont,fill=yellow](v1){$v_1$};
    \node[latent,below right = 2cm of v0,font=\fontsize{20}{0}\selectfont,fill=yellow](v2){$v_2$};
    \node[latent,above right = 4cm and 3cm of v1,font=\fontsize{20}{0}\selectfont](v3){$v_3$};
    \node[latent,above right = 1cm and 3cm of v1,font=\fontsize{20}{0}\selectfont,fill=orange](v4){$v_4$};
    \node[latent,above right = 1cm and 3cm of v2,font=\fontsize{20}{0}\selectfont](v5){$v_5$};
    \node[latent,below right = 2cm and 3cm of v2,fill=orange,font=\fontsize{20}{0}\selectfont](v6){$v_6$};
    \node[latent,scale=0.5,above right = 2cm and 3cm of v3](s1){};
    \node[latent,scale = 0.5, right = 3cm of v3](s2){};
    \node[latent,above right = 1cm and 3cm of v4,scale=00.5](v7){};
    \node[latent,right = 3cm of v4,font=\fontsize{20}{0}\selectfont,fill=green](v8){$v_7$};
    \node[latent,above right = 2cm and 3cm of v5,scale=0.5](v9){};
    \node[latent,right =4cm of v5,scale=0.5](v10){};
    \node[latent,above right = 2cm and 3cm of v6,font=\fontsize{20}{0}\selectfont](v11){$v_{8}$};
    \node[latent,below right = 1cm and 3cm of v6,font=\fontsize{20}{0}\selectfont,fill=green](v12){$v_{9}$};
    \node[latent,above right = 2cm and 3cm of v8,scale=0.5](s3){};
    \node[latent, right =2cm of v8,scale=0.5](s4){};
    \node[latent,above right = 2cm and 3cm of v11,scale = 0.5](s7){};
    \node[latent, right =  3cm of v11,scale = 0.5](s8){};
    \node[latent,above right = 2cm and 3cm of v12,scale = 0.5](s9){};
    \node[latent,right = 3cm of v12,scale = 0.5](s10){};
    \node[below = 2cm of v6,font=\fontsize{27}{0}\selectfont](t1){(a) The missingness staged tree};
    \node[latent,font=\fontsize{20}{0}\selectfont, above right = 1cm and 2cm of s7](w0){$w_0$};
    \node[latent,above right = 2cm and 2cm of w0,font=\fontsize{20}{0}\selectfont,fill=yellow](w1){$w_1$};
    \node[latent,below right = 2cm and 2cm of w0,font=\fontsize{20}{0}\selectfont,fill=yellow](w2){$w_2$};
    \node[latent,above right = 2cm and 2cm of w1,font=\fontsize{20}{0}\selectfont](w3){$w_3$};
    \node[latent,below right = 1cm and 2cm of w1,font=\fontsize{20}{0}\selectfont,fill=orange](w4){$w_4$};
    \node[latent,above right = 1cm and 2cm of w2,font=\fontsize{20}{0}\selectfont](w5){$w_5$};
    \node[latent,below right = 2cm and 2cm of w2,font=\fontsize{20}{0}\selectfont,fill=orange](w6){$w_6$};
    \node[latent,above right = 1cm and 3cm of w6,font=\fontsize{20}{0}\selectfont](w7){$w_7$};
    \node[latent,below right = 1cm and 4cm of w6,font=\fontsize{20}{0}\selectfont,fill=green](w8){$w_8$};
    \node[latent,above right = 3cm and 2cm of w8,font=\fontsize{20}{0}\selectfont](wf){$w_\infty^f$};
    \node[latent,above = 6cm of wf,font=\fontsize{20}{0}\selectfont](wn){$w_\infty^n$};
    \node[below = 3cm of w6,font=\fontsize{27}{0}\selectfont](t2){(b) The M-CEG};
    
    
     \node[latent,below right = 3cm and 2cm of wn,font=\fontsize{20}{0}\selectfont](f0){$w_0$};
    \node[latent,above right = 2cm and 2cm of f0,font=\fontsize{20}{0}\selectfont,fill=yellow](f1){$w_1$};
    \node[latent,below right = 2cm and 2cm of f0,font=\fontsize{20}{0}\selectfont,fill=yellow](f2){$w_2$};
    \node[latent,below right = 1cm and 2cm of f1,font=\fontsize{20}{0}\selectfont,fill=orange](f4){$w_4$};
    \node[latent,below right = 2cm and 2cm of f2,font=\fontsize{20}{0}\selectfont,fill=orange](f6){$w_6$};
    \node[latent,below right = 2cm and 2cm of f6,font=\fontsize{20}{0}\selectfont,fill=green](f8){$w_8$};
    \node[latent,above right = 3cm and 2cm of f8,font=\fontsize{20}{0}\selectfont](wf1){$w_\infty^f$};
    \node[latent,above = 6cm of wf1,font=\fontsize{20}{0}\selectfont](wn1){$w_\infty^n$};
    \node[below = 3cm of f6,font=\fontsize{27}{0}\selectfont](t3){(c) The intervened M-CEG};
    \path[->,draw]
    (v0)edge node[scale=2.5]{\text{1 off}} (v1)
    (v0)edge node[scale=2.5]{\text{1 on}} (v2)
    (v1)edge[color=brown] node[scale=2.5]{Missing} (v3)
    (v1)edge[color=gray] node[scale=2.5]{NM} (v4)
    (v2)edge[color=brown] node[scale=2.5]{Missing} (v5)
    (v2)edge[color=gray] node[scale=2.5]{NM} (v6)
    (v3)edge node[scale=2.5]{no fault}(s1)
    (v3)edge node[scale=2.5]{fault}(s2)
    (v4)edge[color=red] node[scale=2.5]{2 off} (v7)
    (v4)edge[color=blue] node[scale=2.5]{2 on} (v8)
    (v5)edge node[scale=2.5]{no fault} (v9)
    (v5)edge node[scale=2.5]{fault} (v10)
    (v6)edge[color=red] node[scale=2.5]{2 off} (v11)
    (v6)edge[color=blue] node[scale=2.5]{2 on} (v12)
    (v8)edge[color=green] node[scale=2.5]{no fault}(s3)
    (v8)edge[color=purple] node[scale=2.5]{fault} (s4)
    (v11)edge node[scale=2.5]{no fault}(s7)
    (v11)edge node[scale=2.5]{fault}(s8)
    (v12)edge[color=green] node[scale=2]{no fault}(s9)
    (v12)edge[color=purple] node[scale=2.5]{fault}(s10)
   
    
     (w0)edge[bend left=15] node[scale=2.5]{\text{1 off}} (w1)
    (w0)edge[bend right=15] node[scale=2.5]{\text{1 on}} (w2)
    (w1)edge[bend left=15,color=brown] node[scale=2.5]{\text{Missing}}(w3)
    (w1)edge[bend right=15,color=gray] node[scale=2.5]{\text{NM}}(w4)
    (w2)edge[bend left=15,color=brown] node[scale=2.5]{\text{Missing}}(w5)
    (w2)edge[bend right=15,color=gray] node[scale=2.5]{\text{NM}}(w6)
    (w3)edge[bend left=25] node[scale=2]{\text{no fault}}(wn)
    (w3)edge[bend left=15] node[scale=2]{\text{fault}}(wf)
    (w4)edge[bend left=15,color=red] node[scale=2.5]{\text{2 off}}(wn)
    (w4)edge[bend left=15,color=blue] node[scale=2.5]{\text{2 on}}(w8)
    (w5)edge[bend left=15] node[scale=2]{\text{no fault}}(wn)
    (w5)edge[bend left=15] node[scale=2]{\text{fault}}(wf)
    (w6)edge[bend left=15,color=red] node[scale=2.5]{\text{2 off}}(w7)
    (w6)edge[bend right=15,color=blue] node[scale=2.5]{\text{2 on}}(w8)
    (w7)edge[bend left=15] node[scale=2]{\text{no fault}}(wn)
    (w7)edge[bend right=15] node[scale=2]{\text{fault}}(wf)
    (w8)edge[bend left=15,color=green] node[scale=2.5]{\text{no fault}}(wn)
    (w8)edge[bend right=15,color=purple] node[scale=2]{\text{fault}}(wf)
    (f0)edge[bend left=15] node[scale=2.5]{\text{1 off}} (f1)
    (f0)edge[bend right=15] node[scale=2.5]{\text{1 on}} (f2)
    (f1)edge[bend right=15,color=gray] node[scale=2.5]{\text{NM}}(f4)
    (f2)edge[bend right=15,color=gray] node[scale=2.5]{\text{NM}}(f6)
    (f4)edge[bend left=15,color=blue] node[scale=2.5]{\text{2 on}}(f8)
    (f6)edge[bend right=15,color=blue] node[scale=2.5]{\text{2 on}}(f8)
    (f8)edge[bend left=15,color=green] node[scale=2.5]{\text{no fault}}(wn1)
    (f8)edge[bend right=15,color=purple] node[scale=2.5]{\text{fault}}(wf1)
    }}
    \caption{The missingness staged tree and the corresponding CEG for the example in Section \ref{sec:missing}. The stages in the missingness staged tree are $\{v_0\},\{v_1,v_2\},\{v_3\},\{v_4,v_6\},\{v_5\},\{v_7,v_9\},\{v_8\}$. Subgraph (c) is used for recovering the causal identifiability.}
    \label{fig:example m-ceg}
\end{figure}

\begin{example}\label{missing example}
We use an example in \citep{Thwaites13} to demonstrate how to use the M-back-door theorem for a singular manipulation on an M-CEG. Figure \ref{fig:example f-ceg}(a) depicts the staged tree for this system which has two warning lights that indicate possible faults in two components. There are four stages in this tree: $\{v_0\},\{v_1,v_2\},\{v_3,v_5\},\{v_4\}$, where $v_3$ and $v_5$ are in the same position. Figure \ref{fig:example f-ceg}(b) is the fact CEG elicitepd from the fact staged tree. The stages are $u_1=\{w_0\},u_2=\{w_1,w_2\},u_3=\{w_3\},u_4=\{w_4\}$.

Suppose the florets $F(v_1)$ and $F(v_2)$ are unobservable and they have the same probability of being missing. Then the missingness staged tree is constructed in Figure \ref{fig:example m-ceg}(a). Vertices that are at the same stage in the fact staged tree are still at the same stage in the missingness staged tree. So we can elicit an M-CEG from the M-staged tree, see Figure \ref{fig:example m-ceg}(b). The stages are $u_1^M=\{w_0\},u_2^M=\{w_1,w_2\},u_3^M=\{w_3\},u_4^M=\{w_4,w_6\},u_5^M=\{w_5\},u_6^M=\{w_7\},u_7^M=\{w_8\}$.

If we force the second light to be on, that is ``2 on'' must be observed, then on a F-CEG, when we are interested in machine's failure, we can choose the back-door partition $\Lambda_{w(z)}$ so that $w(z)\in u_2$. Thus
\begin{equation}\label{equ:int}
    \begin{split}
        \pi(\Lambda_{w_\infty^f}||\Lambda_{w_4}) =\pi(\Lambda_{w_\infty^f}|\Lambda_{w_4},\Lambda_{w_1})\pi(\Lambda_{w_1}) + \pi(\Lambda_{w_\infty^f}|\Lambda_{w_4},\Lambda_{w_2})\pi(\Lambda_{w_2}).
    \end{split}
\end{equation}

On the M-CEG, we can still use the indicator of whether the first light is on to create the back-door partition, where $\Lambda_z=\Lambda_{e_{w_0,w_1}}\cup\Lambda_{e_{w_0,w_2}}$. Note that this indicator is always observed. We can now draw an intervened M-CEG where the controlled event, the effect event and the block set are all observable. This is shown in Figure \ref{fig:example m-ceg}(c). Then equation \ref{equ:int} can be recovered.
\end{example}

\section{Discussion} There are of course other classes of model that could further refine the process we describe in this paper. An exciting possibility is to study an alternative hierarchical model where the familiar BN textual extraction is replaced by a more expressive CEG expression. Early results on this suggest that the best scoring CEG outperforms the BN \citep{Barclay2013,Cowell14} and addresses the problem of data sparsity. However we need to carefully define the variables and the ordering of variables to avoid creating a CEG with an opaque topology and a slow selection algorithm since the search space across CEGs is of orders of magnitude greater than that of BNs.

A current priority is for us to produce algorithms that enable us to scale up the code of the methodology described here so that the massive reports within this domain can be quickly used to refine inference about system faults. Intervention calculi can be designed for other types of maintenance existed in reliability, for example, we will show how to identify the causal effects in the routine intervention regime in our recent paper \citep{Yu2021}.

Despite these ongoing challenges we believe potential of formal methods for representing and assessing the magnitude of causal effects based on reports such as those produced by engineers are enormous. We note that although the technologies we have described here have been customised to the needs of system reliability many of the concepts and algorithms translate to other domain including the systematic study of electronic health records (EHR) for medication purpose. The effects of medical interventions can be analysed by paralleling the intervention calculus introduced in this paper in a way we will describe in forthcoming work.

\bibliographystyle{unsrtnat}


\end{document}